\documentclass[11pt,reqno]{amsart}
\usepackage{amssymb}
\usepackage{amsmath}
\usepackage{amsthm}
\usepackage{graphicx}
\usepackage{caption}
\usepackage{bm}
\usepackage{subfigure}
\usepackage[colorlinks=true,urlcolor=blue,
citecolor=red,linkcolor=blue,linktocpage,pdfpagelabels,
bookmarksnumbered,bookmarksopen]{hyperref}
\usepackage[titletoc,title]{appendix}
\numberwithin{equation}{section} 
\newcounter{cont}[section] 

\newtheorem{thm}[cont]{Theorem}
\newtheorem{prop}[cont]{Proposition}
\newtheorem{lem}[cont]{Lemma}

\theoremstyle{definition}
\newtheorem{defn}[cont]{Definition}
 \theoremstyle{remark}
 \newtheorem{rem}[cont]{Remark}

\newcommand{\R}{\mathbb{R}}
\newcommand{\e}{\varepsilon}

\begin{document}

\title[Slow motion of interface layers for Allen-Cahn equation]{Exponentially slow motion of interface layers for 
the one-dimensional Allen-Cahn equation with nonlinear phase-dependent diffusivity}

\author[R. Folino]{Raffaele Folino}
\address[R. Folino]{Departamento de Matem\'aticas y Mec\'anica\\Instituto de Investigaciones en Matem\'aticas Aplicadas y en 
Sistemas\\Universidad Nacional Aut\'{o}noma de M\'{e}xico\\Circuito Escolar s/n, Ciudad Universitaria, C.P. 04510\\Cd. de M\'{e}xico (Mexico)}
\email{folino@mym.iimas.unam.mx}

\author[C. A. Hern\'andez Melo]{C\'esar A. Hern\'andez Melo}
\address[C. A. Hern\'andez Melo]{Departamento de Matem\'atica\\Universidade Estadual de Maring\'a\\Av. Colombo, 
5790 Jd. Universit\'ario, CEP 87020-900, Maring\'a, PR (Brazil)}
\email{cahmelo@uem.br}

\author[L. F. L\'{o}pez R\'{\i}os]{Luis F. L\'{o}pez R\'{\i}os}
\address[L. F. L\'{o}pez R\'{\i}os]{Departamento de Matem\'aticas y Mec\'anica\\Instituto de Investigaciones en Matem\'aticas Aplicadas y 
en Sistemas\\Universidad Nacional Aut\'{o}noma de M\'{e}xico\\Circuito Escolar s/n, Ciudad Universitaria, C.P. 04510\\Cd. de M\'{e}xico (Mexico)}
\email{luis.lopez@iimas.unam.mx}
	
\author[R. G. Plaza ]{Ram\'{o}n G. Plaza}
\address[R. G. Plaza]{Departamento de Matem\'aticas y Mec\'anica\\Instituto de Investigaciones en Matem\'aticas Aplicadas 
y en Sistemas\\Universidad Nacional Aut\'{o}noma de M\'{e}xico\\Circuito Escolar s/n, Ciudad Universitaria, C.P. 04510\\Cd. de M\'{e}xico (Mexico)}
\email{plaza@mym.iimas.unam.mx}
 
\keywords{nonlinear diffusion; metastability; energy estimates}
\subjclass[2010]{35K20, 35K57, 35B36, 82B26}

\maketitle


\begin{abstract} 
This paper considers a one-dimensional generalized Allen-Cahn equation of the form
\[
u_t = \e^2 (D(u)u_x)_x - f(u),
\] 
where $\e > 0$ is constant, $D = D(u)$ is a positive, uniformly bounded below diffusivity coefficient that depends on 
the phase field $u$ and $f(u)$ is a reaction function that can be derived from a double-well potential with minima at two pure phases $u = \alpha$ and $u = \beta$. 
It is shown that interface layers (namely, solutions that are equal to $\alpha$ or $\beta$ except at a finite number of thin transitions of width $\e$) 
persist for an exponentially long time proportional to $\exp(C/\e)$, where $C > 0$ is a constant. 
In other words, the emergence and persistence of \emph{metastable patterns} for this class of equations is established. 
For that purpose, we prove energy bounds for a renormalized effective energy potential of Ginzburg-Landau type. 
Numerical simulations, which confirm the analytical results, are also provided.
\end{abstract}


\section{Introduction}\label{sec:intro}
\subsection{The Allen-Cahn model}
The classical Allen-Cahn equation \cite{AlCa79} (also known as the time-dependent Ginz\-burg-Landau equation \cite{Schmd66}),
\begin{equation}\label{eq:cAC}
	u_t = \e^2 D_0 \Delta u - F'(u), \quad x \in \Omega, \: t > 0,
\end{equation}
was introduced to model the motion of spatially non-uniform phase structures in crystalline solids. 
It describes the state of a system confined in a bounded space domain $\Omega \subset \R^n$, $n \geq 1$, 
in terms of a scalar phase field, $u = u(x,t)$ (also called ``order parameter" \cite{GreHu64}), 
depending on space and time variables, $x \in \Omega$ and $t > 0$, respectively, which interpolates two homogeneous pure components, 
$u = \alpha$ and $u = \beta$, of the binary alloy. 
The potential $F\in C^3(\R)$ is a prescribed function of the phase field $u$, 
having a double-well shape with local minima at the preferred $\alpha$- and $\beta$-phases. 
The parameter $\e>0$ measures the interface width separating the phases and $D_0>0$ is a constant diffusion coefficient (also known as \emph{mobility}). 
Associated to equation \eqref{eq:cAC} is the Ginzburg-Landau free energy functional,
\begin{equation}\label{eq:GLf}
	\overline{E}_\e[u] = \int_{\Omega} \left\{ \frac{1}{2} \e^2 D_0 |\nabla u|^2 + F(u) \right\} \, dx.
\end{equation}
The free energy per unit volume has two contributions:
$F(u)$ is the free energy that a small volume would have in an homogeneous concentration with value $u$, 
whereas the term $\tfrac{1}{2} \e^2 D_0 |\nabla u|^2$ penalizes spatial variation with an energy cost associated to an interface between the two pure phases. 
In general, this ``gradient energy"  density is also a function of the local composition \cite{CaHi58,Cahn59}. 
It is well-known (see, e.g., \cite{CISc13,Fif02} and the references therein) that, 
if no constraints are imposed on the total value of the phase field $u$ in $\Omega$, 
then the $L^2(\Omega)$-gradient flow of the Ginzburg-Landau functional \eqref{eq:GLf} is the Allen-Cahn equation \eqref{eq:cAC} 
endowed with homogeneous Neumann boundary conditions,
\begin{equation}\label{eq:Nbc}
	\partial_\nu u = 0, \quad \text{on } \; \partial \Omega,
\end{equation}
(where $\nu \in \R^n$, $|\nu| = 1$, is the outer unit normal at each point of $\partial \Omega$) describing no flux of atoms outside the physical domain $\Omega$. 
If the integral of $u$ is assumed to be constant, then the $H^{-1}(\Omega)$-gradient flow of \eqref{eq:GLf} results into the well-known Cahn-Hilliard model \cite{CaHi58,CISc13}.

When the parameter $0<\e\ll1$ is very small, the order parameter $u$ concentrates near the pure components $\alpha$ and $\beta$, 
except possibly at sharp anti-phase boundaries known as \emph{interface (or phase transition) layers}. 
The limit when $\e\to 0^+$ describes such sharp interface layers separating both phases inside an heterogeneous composition of the alloy. 
In the study of their dynamics, the renormalized Ginzburg-Landau energy functional
\begin{equation}\label{eq:GLf}
	E_\e[u] = \frac{1}{\e} \overline{E}_\e[u] = \int_{\Omega} \left\{ \frac{1}{2} \e D_0 |\nabla u|^2 + \frac{F(u)}{\e} \right\} \, dx,
\end{equation}
plays a key role. 
For instance, sequences with uniformly renormalized energy converge to a function which is equal to $\alpha$ or $\beta$ a.e. 
with a finite-perimeter interface $\partial\{ u=\alpha, u=\beta\}$ (see, e.g., \cite{OwSt91} and the discussion in \cite{KORV07}). 
In the framework of $\Gamma$-convergence theory introduced by De Giorgi \cite{DeGFr75} (see also \cite{Brai02,DalM93}), 
Modica and Mortola \cite{MoMo77a,MoMo77b} showed, for instance, that the $\Gamma(L^1(\Omega))$-limit as $\e\to 0^+$ 
of the renormalized energy functional ${E}_\e[\cdot]$ is proportional to the \emph{perimeter functional},
\begin{equation*}
	E_\e[\cdot] \stackrel{\Gamma}{\longrightarrow} \gamma_0 \mathrm{Per}_\Omega (A_{\alpha \beta}):=
	\gamma_0 \mathcal{H}^{n-1} (\partial A_{\alpha \beta} \cap \Omega), \qquad \e \to 0^+,
\end{equation*}
where $A_{\alpha \beta} = \{ x \in \Omega \, : \, u(x) = \alpha, u(x) = \beta \}$, 
$\mathcal{H}^{n-1}$ is the $n-1$ dimensional Hausdorff measure \cite{Fede69} and $\gamma_0 > 0$ is a constant independent of $\e > 0$, 
an energy lower bound for the renormalized energy. 
In several space dimensions ($n > 1$), it is well known that this fast reaction/slow diffusion regime leads to limiting dynamics of \eqref{eq:cAC} 
when $\e \to 0^+$ in which the interface $\partial A_{\alpha \beta}$ evolves by mean curvature flow (cf. \cite{EvSS92,Ilm93,RSK89}).

In one space dimension ($n=1$), in contrast, the situation is rather peculiar. 
Since here interface layers are points, the Hausdorff measure $\mathcal{H}^0$ reduces to the counting measure in $\R$ and 
the finite perimeter is simply the number $N$ of transitions in the layer structure. 
This results into a negligible effect of the interfaces on the main term of the energy and the interface motion is exponentially slow. 
This phenomenon is known in the literature as \emph{metastability} \cite{BrKo90,CaPe89,CaPe90,FuHa89}: 
interface layers are transient solutions that appear to be stable but which, after an exponentially long time of order $T_\e = O(\exp (1/\e))$, drastically change their shape. 
In this work we establish this (one-dimensional) phenomenon in the case where the diffusion coefficient depends on $u$.

\subsection{Phase-dependent diffusivity}
In some physical situations the diffusion coefficient or mobility should be regarded as a function of the phase field. 
It is to be observed that, in fact, a concentration-dependent diffusional mobility appears in the original derivation of the Ginzburg-Landau energy functional 
by Cahn \cite{Cahn61} and Cahn and Hilliard \cite{CaHi71} (see also \cite{CEN-C96,TaCa94}). 
Such dependence has been incorporated into several mathematical models of Allen-Cahn or Cahn-Hilliard type 
to describe a great variety of physical systems (an abridged list of references include \cite{CISc13,CISc16,DaiDu16b,DaiDu16a,DGN-C99,EllGa96}). 
Particularly, in the physics literature there exist one-dimensional, phase-transitional models with concentration-dependent, 
strictly positive diffusivities which are described by equations of the form  
\begin{equation}\label{eq:Dform}
	u_t = (D(u) u_x)_x - f(u),
\end{equation}
such as the experimental exponential diffusion function for metal alloys (cf. Wagner \cite{Wagn52}) and the Mullins diffusion model for thermal grooving \cite{Broa89,Mull57}.

The first example pertains to the description of the physical properties of metal alloys, 
for which it is well-known that the characteristic lattice parameter varies with the metal composition \cite{HaCl49,Wagn52}. 
Hence, it is customary to propose a phenomenological diffusion profile based on the Boltzmann-Matano method \cite{Boltz1894,Mtn33},
which allows to approximate the diffusion coefficient as a function of concentration over the whole range of composition from one phase to the other. 
This function can be computed from experimental results. 
Wagner \cite{Wagn52} argues that the function that best fits many experiments on binary alloys 
(such as Fe-C, Cu-Zn, Cu-Al and Fe-Al, just to mention a few; see \cite{Wagn52} and the experimental references mentioned therein) 
has the shape of an exponential function
\begin{equation}\label{eq:expDiff}
	D = D(u) = D_0 \exp \left\{c_0\left(u - \tfrac{1}{2}(\alpha + \beta)\right)\right\},
\end{equation}
where $c_0>0$ is an experimental constant and $D_0>0$ is the value of the diffusion at the Matano boundary (where the concentration is the average of the two pure phases). Temperature and pressure are supposed to be constant and changes normal to the direction of diffusion are negligible, 
so that a one-dimensional model is usually applicable (see, e.g., \cite{Lee71}). 
When coupled with a reaction process based on a double-well potential as Allen and Cahn \cite{AlCa79} originally proposed for a Fe-Al metal binary alloy, 
for instance, one recovers an equation of the form \eqref{eq:Dform}. 

The second example refers to thermal grooving, that is, the development of surface groove profiles on a heated polycrystal 
by the mechanism of evaporation-condensation (cf. Mullins \cite{Mull57}, Broabridge \cite{Broa89}). 
The metal polycrystal is assumed to be in quasi-equilibrium of its vapor, and the surface diffusion process, 
as well as the mechanism of evaporation and condensation, are modelled via Gibbs-Thompson formula. 
In view that the properties of the interface do not depend on its orientation, this is essentially a one-dimensional phenomenon. 
After an appropriate transformation (see \cite{Broa89}), the Mullins nonlinear diffusion model of groove development 
can be expressed in terms of a one-dimensional diffusion equation where the nonlinear diffusion coefficient is given by
\begin{equation}\label{eq:MullinsD}
	D(u) = \frac{D_0}{1+ u^2},
\end{equation}
with $D_0>0$, constant. 
Notice that $D>0$ in its whole domain.
When the Mullins pure diffusion model is coupled to a Gibbs energy double-well potential $F$ for the two polycrystal evaporation-condensation phases, 
the result is an reaction-diffusion equation with nonlinear phase-dependent diffusivity of the form \eqref{eq:Dform}.

\subsection{Assumptions and main results}
The goal of this paper is to study the emergence and persistence of metastable phase transition layers 
in a one-dimensional generalized Allen-Cahn model with a phase-dependent diffusivity coefficient of the following form
\begin{equation}\label{eq:D-model}
	u_t=\e^2(D(u)u_x)_x-f(u), \qquad \qquad x\in  (a,b), \; t>0,
\end{equation}
endowed with homogeneous Neumann boundary conditions
\begin{equation}\label{eq:Neu}
	u_x(a,t)=u_x(b,t)=0, \quad \quad \; t>0,
\end{equation}
and initial datum
\begin{equation}\label{eq:initial}
	u(x,0)=u_0(x), \qquad \qquad x\in[a,b].
\end{equation}
Here $\e>0$ is a small parameter, the diffusivity coefficient $D = D(u)$ is a \emph{strictly positive} function of the phase field and 
the reaction term $f=f(u)$ is assumed to be of \emph{bistable} type. More precisely, we assume that there exists an open interval $I\subset\R$ such that $D\in C^2(I)$ satisfies 
\begin{equation}\label{eq:ass-D}
	D(u)\geq d>0, \qquad \qquad \forall\, u\in I,
\end{equation}
and $f\in C^2(I)$ is such that
\begin{equation}\label{eq:ass-f}
	f(\alpha)=f(\beta)=0, \qquad \qquad f'(\alpha)>0, \quad f'(\beta)>0,
\end{equation}
for some $\alpha<\beta$ with $[\alpha,\beta]\subset I$. 
Clearly, if we choose $D(u)\equiv D_0 > 0$, constant, then we recover the classical Allen-Cahn equation \eqref{eq:cAC} in one dimension 
\begin{equation}\label{eq:Al-Ca1d}
	u_t=\e^2 D_0 u_{xx} -F'(u),
\end{equation}
with potential $F(u) = \int^u f(s) \, ds$.

Crucial assumptions in this work concerning the interaction between $f$ and $D$ are
\begin{equation}\label{eq:ass-int0}
	\int_\alpha^\beta f(s)D(s)\,ds=0, 
\end{equation}	
and
\begin{equation}\label{eq:ass-int1}
	\int_\alpha^u f(s)D(s)\,ds>0, \qquad \forall\,u\neq\alpha,\beta.
\end{equation}
It is important to notice that \eqref{eq:ass-D}-\eqref{eq:ass-f}-\eqref{eq:ass-int0}-\eqref{eq:ass-int1} imply that the function
\begin{equation}\label{eq:G}
	G(u):=\int_{\alpha}^{u}f(s)D(s)\, ds
\end{equation}
is an effective double-well potential with wells of equal depth, i.e. $G:I\to\R$ satisfies 
\begin{equation}\label{eq:ass-G}
	\begin{aligned}
		G(\alpha)=G(\beta)=G'(\alpha)&=G'(\beta)=0, \qquad G''(\alpha)>0,\quad G''(\beta)>0, \\ 
		&G(u)>0. \quad \forall\, u\neq\alpha,\beta,
	\end{aligned}
\end{equation}

Equation \eqref{eq:D-model} is a parabolic equation in $I$ in view of the assumption \eqref{eq:ass-D} and 
throughout the paper we consider initial data satisfying
\begin{equation}\label{eq:u0continI}
	M_1\leq u_0(x)\leq M_2, \qquad \qquad \forall\,x\in[a,b],
\end{equation}
for some $M_1<M_2$ such that $[M_1,M_2]\subset I$.
Therefore, we can use the classical maximum principle to conclude that the solution $u$ to the initial boundary value problem
 \eqref{eq:D-model}-\eqref{eq:Neu}-\eqref{eq:initial} remains in $I$ for all time $t\geq0$.

\begin{rem}\label{rem:potential}
The typical example of a reaction function with a bistable structure is the cubic polynomial of the form
\[
	f(u) = (u - \beta)(u - \alpha)(u - u_*),
\]
where $u _* \in (\alpha, \beta)$ is an energy barrier where the potential (up to a constant)
\[
	F(u) = \int^u f(s) \, ds,
\]
has a local maximum. 
Since the minimization of \eqref{eq:GLf} remains unchanged (up to a constant) when $F$ is replaced by the affine transformation 
$F(u) \to F(u) - (c_1 u + c_0)$, in the standard Ginzburg-Landau theory one usually finds energy potentials of the form
\[
	F(u) = \frac{1}{4}(u - \alpha)^2 (u - \beta)^2,
\]
for which the energy barrier lies precisely at the midpoint $u_* = \tfrac{1}{2}(\alpha + \beta)$ 
(it is customary in the literature to consider $\alpha = -1$, $\beta = 1$, yielding $u_* = 0$). 
This produces two energy wells with same minimum value ($F(\alpha)=F(\beta)=0$), symmetrically located with respect to the barrier in between. 
A necessary condition for this to happen is that
\[
	\int_\alpha^\beta f(s) \, ds = 0.
\]
In our case, this condition is the equivalent to \eqref{eq:ass-int0} and, thus, the effective, 
diffusion-weighted energy density $G$ inherits the properties of an equal-well potential. 
In this case, though, the energy barrier might not be symmetrically located depending on the properties of the diffusion function $D$.
\end{rem}

Motivated by previous metastability results for the classical Allen-Cahn equation \eqref{eq:Al-Ca1d} (see, for example, \cite{BrKo90,CaPe89,CaPe90,FuHa89}), 
in this paper we apply the energy approach of Bronsard and Kohn \cite{BrKo90} to rigorously prove 
the existence of metastable states for the initial boundary-value problem (IBVP) \eqref{eq:D-model}-\eqref{eq:Neu}-\eqref{eq:initial}. 
We recall that in \cite{BrKo90} the authors introduce the energy approach to prove persistence of metastable patterns for \eqref{eq:Al-Ca1d} 
for a time $\mathcal{O}(\e^{-k})$ for any $k\in\mathbb{N}$.
Then, the energy approach was improved in \cite{Grnt95}, 
where the author obtains persistence the exponentially long time $\mathcal{O}(\exp\left(-C/\e)\right)$ in the case of Cahn-Morral systems.
By using these ideas, the energy approach has been applied to prove persistence of metastable patterns in many different models, 
see among others \cite{Fol19,FLM19} and references therein.

The main goal of this paper is to show how to adapt the energy approach to the case of a nonlinear diffusion \eqref{eq:D-model}.
For that purpose, we study the following renormalized energy functional
\begin{equation}\label{eq:energy}
	E_\e[u]=\int_a^b\left\{\frac\e2[D(u)u_x]^2+\frac{G(u)}\e\right\}\,dx,
\end{equation}
where $G$ is the effective potential defined in \eqref{eq:G}. 
Notice that the $u$-dependent mobility is involved not only in the gradient term of the energy density, but also in the barrier energy term via the function $G$. 
We regard this energy functional as a \emph{generalized effective energy of Ginzburg-Landau type}.

Let us now present a panoramic overview of the results of this paper. 
The main theorem (Theorem \ref{thm:main} below) establishes that, if an initial datum for equation \eqref{eq:D-model} has an $N$-\emph{transition layer structure} 
(for the precise definition see Definition \ref{def:TLS} below) then the IBVP \eqref{eq:D-model}-\eqref{eq:Neu}-\eqref{eq:initial} maintains this structure for an exponentially long time. 
For that purpose, we prove (see Proposition \ref{prop:lower}) the following variational result for the generalized energy \eqref{eq:energy}: 
if $u \in H^1(a,b)$ underlies an $N$-transition layer structure then there exist constants $A, C, \delta$ and $\e_0>0$, independent of $u$, such that
\begin{equation}\label{eq:sharpest}
	E_\e[u] \geq N \gamma_0 - C \exp(-A/\e),
\end{equation}
for all $\e \in (0,\e_0)$. 
Here, $\gamma_0>0$ is a constant that depends on the energy $F$ and the diffusion coefficient alone (see \eqref{eq:gamma} below). 
Such energy bound is reminiscent (actually, strongly related) to the $\Gamma$-convergence of the functional, 
inasmuch as the perimeter functional is the number of transitions $N$ and the constant $\gamma_0$ is independent of $\e>0$. 
Estimate \eqref{eq:sharpest}, however, is inherently sharper, because it provides exact information on the convergence rate as $\e \to 0^+$. 
The exponential term in \eqref{eq:sharpest} is crucial to show the exponentially slow motion of the interfaces and, 
in turn, to prove our main result, Theorem \ref{thm:main}. 
These results are the content of Section \ref{sec:met}. 

\begin{rem}\label{rem:Gamma-convergence}
In \cite{OwSt91}, the authors study the $\Gamma$-convergence properties of the general class of functionals 
\begin{equation*}
	W_\e[u]:=\int_{\Omega}\frac{1}{\e}w(x,u,\e\nabla u)\,dx,
\end{equation*}
where $\Omega$ is an open, bounded domain in $\R^n$ and the function $w\in C^3(\bar{\Omega}\times\R\times\R^n)$ 
satisfies appropriate assumptions (cfr. assumptions (H1)-(H7) in \cite{OwSt91}).
In particular, the generalized energy \eqref{eq:energy} with $D$ and $G$ satisfying \eqref{eq:ass-D}-\eqref{eq:ass-G} enters in the framework of \cite{OwSt91}, 
but as it was already mentioned, the $\Gamma$-convergence properties of \cite{OwSt91} are not sufficient 
to prove the exponentially slow motion of the solutions to \eqref{eq:D-model}-\eqref{eq:Neu}-\eqref{eq:initial} and we need the sharper inequality \eqref{eq:sharpest}.
We also stress that \eqref{eq:sharpest} holds for generic functions $D$ and $G$ satisfying \eqref{eq:ass-D}-\eqref{eq:ass-G} 
($G$ can be different from \eqref{eq:G}) and that the lower bound \eqref{eq:sharpest} (valid only in the one-dimensional case) 
can be used to find the $\Gamma(L^1(a,b)^-)$-limit of \eqref{eq:energy} (by proceeding as in \cite{OwSt91}).
Indeed, Proposition \ref{prop:lower} implies that for any sequence $u^{\e}$ converging in $L^1$ to a step function $v$,
which assumes only the values $\alpha,\beta$ and with exactly $N$ jumps, one has
\begin{equation*}
	\liminf_{\e\to0^+}E_\e[u^{\e}]\geq N\gamma_0.
\end{equation*}
Moreover, we can properly construct a sequence $u^{\e}$ such that the equality holds (see Proposition \ref{prop:ex-met}).
\end{rem}

Section \ref{sec:exist} is devoted to show the existence of metastable patterns. 
We construct a family of functions having an $N$-transition layer structure (see Proposition \ref{prop:ex-met}) 
based on standing wave solutions to \eqref{eq:D-model} with asymptotic boundary conditions. 
In addition, we provide an upper bound for the velocity of the transitional points of the solution (see Theorem \ref{thm:interface}), 
conveying a precise characterization of the dynamics of the interfaces. 
This estimation relies on a purely variational result (Lemma \ref{lem:interface} below) which states that, 
if $u$ underlies an $N$-transition layer structure and $E_\e[u]$ slightly exceeds the minimum energy to have $N$ transitions, 
then the Hausdorff distances between the interface associated to $u$ and the $N$ transition structure is arbitrarily small. 
We finish the paper by presenting the results of some numerical simulations with diffusivities of the form $D=D_0$, constant (the classical Allen-Cahn model), 
the exponential function \eqref{eq:expDiff} and Mullins diffusion \eqref{eq:MullinsD}. 
These numerical experiments confirm the analytical results (see Section \ref{sec:num} below). 
Furthermore, numerical simulations with \emph{degenerate} diffusion (e.g., $D$ vanishing at one or both of the pure phases, like for example, a \emph{porous medium} type diffusion \cite{Vaz07}) 
or degenerate reaction (that is, $f'(\alpha)$ or $f'(\beta)$ equal to zero) were also performed. 
These choices were motivated by physical considerations (see the discussion on Section \ref{sec:discussion}).

\section{Persistence of metastable patterns}\label{sec:met}
The main goal of this section is to show that if the initial datum $u_0^\e$ has an $N$-\emph{transition layer structure}, 
then the solution to the IBVP \eqref{eq:D-model}-\eqref{eq:Neu}-\eqref{eq:initial} maintains such a structure for an exponentially long time as $\e\to0^+$.
\begin{defn}\label{def:TLS}
Let us fix $N\in\mathbb{N}$, a {\it piecewise constant function} $v$ with $N$ transitions
\begin{equation}\label{vstruct}
	\begin{aligned}
	v:[a,b]&\rightarrow\{\alpha,\beta\}\  \hbox{with $N$ jumps located at } a<h_1<h_2<\cdots<h_N<b, \\ 
	&\mbox{and } \ r>0 \hbox{ such that } h_i+r<h_{i+1}-r, \ \hbox{ for}\ i=1,\dots,N,\\
	&\qquad\qquad\qquad   a\leq h_1-r,\ h_N+r\leq b.
	\end{aligned}
\end{equation}
Moreover, fix $D$ and $f$ satisfying \eqref{eq:ass-D}, \eqref{eq:ass-f} and \eqref{eq:ass-int0}-\eqref{eq:ass-int1}.
We say that a function $u^\e\in H^1(a,b)$ has an \emph{$N$-transition layer structure} if 
\begin{equation}\label{eq:ass-u0}
	\lim_{\varepsilon\rightarrow 0} \|u^\varepsilon-v\|_{{}_{L^1}}=0,
\end{equation}
and there exist $C>0$ and $A\in(0,r\sqrt{2\lambda})$, where 
\begin{equation}\label{eq:lambda}
	\lambda:=\min\left\{\frac{f'(\alpha)}{D(\alpha)},\frac{f'(\beta)}{D(\beta)}\right\},
\end{equation}
such that
\begin{equation}\label{eq:energy-ini}
	E_\varepsilon[u^\varepsilon]\leq N\gamma_0+C\exp(-A/\e),
\end{equation}
for any $\varepsilon\ll1$, where the energy $E_\e$ is defined \eqref{eq:energy} and the positive constant $\gamma_ 0$ is given by
\begin{equation}\label{eq:gamma}
	\gamma_0:=\int_{\alpha}^{\beta}\sqrt{2G(s)}D(s)\, ds.
\end{equation}
\end{defn}

The first step to prove persistence of $N$-\emph{transition layer structures} for an exponentially long time is 
to study the behavior of the energy \eqref{eq:energy} along the solutions to \eqref{eq:D-model}-\eqref{eq:Neu}.

\begin{lem}\label{lem:energy}
Let $u\in C([0,T],H^2(a,b))$ be solution to equation \eqref{eq:D-model} with homogeneous Neumann boundary conditions \eqref{eq:Neu}.
If $E_\e$ is the functional defined in \eqref{eq:energy}, then 
\begin{equation}\label{eq:energy-dec}
	E_\e[u](0)-E_\e[u](T)=\e^{-1}\int_0^T\int_{a}^{b} D(u)u_t^2\,dxdt.
\end{equation}
\end{lem}
\begin{proof}
Multiplying equation \eqref{eq:D-model} by $D(u)u_t$ and integrating in the interval $[a,b]$, we deduce
\begin{equation*}
	\int_{a}^{b} D(u)u_t^2dx= \int_{a}^{b}\left[\e^2{(D(u)u_x)}_xD(u)u_t-f(u)D(u)u_t\right]dx.
\end{equation*}
Integrating by parts and using the boundary conditions \eqref{eq:Neu}, we get
\begin{equation*}
	\int_{a}^{b} D(u)u_t^2dx= -\int_{a}^{b}\left[\e^2D(u)u_x{(D(u)u_x)}_t+G'(u)u_t\right]dx,
\end{equation*}
where we used the definition \eqref{eq:G}.
Since 
\begin{equation*}
	\e^2D(u)u_x{(D(u)u_x)}_t+G'(u)u_t=\frac{\partial}{\partial t}\left\{\frac{\e^2}2\left[D(u)u_x\right]^2+G(u)\right\},
\end{equation*}	
using the definition of the energy \eqref{eq:energy}, we obtain 
\begin{equation*}
	-\frac{d}{dt}E_\e[u](t)=\e^{-1}\int_{a}^{b} D(u)u_t^2dx,
\end{equation*}
and integrating the latter equality in $[0,T]$ we end up with \eqref{eq:energy-dec}.
\end{proof}
The equality \eqref{eq:energy-dec} holds true for any smooth functions $D,f$; 
in particular, if $D$ satisfies \eqref{eq:ass-D} and the initial datum $u_0$ satisfies \eqref{eq:u0continI}, then we can state that the energy functional \eqref{eq:energy} 
is a non-increasing function of time along the solution to \eqref{eq:D-model}-\eqref{eq:Neu}-\eqref{eq:initial}, namely
\begin{equation}\label{eq:energy-var}
	d\,\e^{-1}\int_0^T\!\!\int_{a}^{b} u_t^2\,dxdt\leq E_\e[u](0)-E_\e[u](T).
\end{equation}
In particular, the latter inequality tells us that if \eqref{eq:energy-ini} is satisfied at time $t=0$, then it holds for any positive time $t$.
Concerning property \eqref{eq:ass-u0}, the main result of this paper states that if the initial datum satisfies \eqref{eq:u0continI} and has a $N$-transition layer structure, 
then the solution to the IBVP \eqref{eq:D-model}-\eqref{eq:Neu}-\eqref{eq:initial} satisfies \eqref{eq:ass-u0} for an exponentially long time, 
and so we can conclude that the solution maintains the same structure of the initial datum for an exponentially long time.
\begin{thm}\label{thm:main}
Assume that $f,D\in C^2(I)$ satisfy \eqref{eq:ass-D}-\eqref{eq:ass-f}-\eqref{eq:ass-int0}-\eqref{eq:ass-int1}.
Let $v$ be as in \eqref{vstruct} and let $A\in(0,r\sqrt{2\lambda})$, with $\lambda$ defined in \eqref{eq:lambda}.
If $u^\varepsilon$ is the solution of \eqref{eq:D-model}-\eqref{eq:Neu}-\eqref{eq:initial} 
with initial datum $u_0^{\varepsilon}$ satisfying \eqref{eq:u0continI}, \eqref{eq:ass-u0} and \eqref{eq:energy-ini}, then, 
\begin{equation}\label{eq:limit}
	\sup_{0\leq t\leq m\exp(A/\varepsilon)}\|u^\varepsilon(\cdot,t)-v\|_{{}_{L^1}}\xrightarrow[\varepsilon\rightarrow0]{}0,
\end{equation}
for any $m>0$.
\end{thm}
The crucial step in the proof of Theorem \ref{thm:main} is to show a particular lower bound on the energy (see \eqref{eq:sharpest}).
Such a result is purely variational in character and the model \eqref{eq:D-model}-\eqref{eq:Neu} plays no role.
As we already discussed in Remark \ref{rem:Gamma-convergence}, the following lower bound, 
which holds for any strictly positive function $D$ and any double well potential $G$, can be used  to study the $\Gamma$-convergence of the functional \eqref{eq:energy}.
\begin{prop}\label{prop:lower}
Assume that $D\in C^1(I)$ satisfies \eqref{eq:ass-D} and that $G\in C^3(I)$ satisfies \eqref{eq:ass-G}. 
Let 
$$
\lambda:=\min\left\{\frac{G''(\alpha)}{D^2(\alpha)}, \frac{G''(\beta)}{D^2(\beta)}\right\}>0,
$$ 
$v$ be as in \eqref{vstruct} and $A\in(0,r\sqrt{2\lambda})$.
Then, there exist $\e_0,C,\delta>0$ (depending only on $G,v$ and $A$) such that if $u\in H^1(a,b)$ satisfies \eqref{eq:u0continI} and
\begin{equation}\label{eq:u-v}
	\|u-v\|_{{}_{L^1}}\leq\delta,
\end{equation}
then for any $\e\in(0,\e_0)$,
\begin{equation}\label{eq:lower}
	E_\varepsilon[u]\geq N\gamma_0-C\exp(-A/\varepsilon),
\end{equation}
where $E_\e$ and $\gamma_0$ are defined in \eqref{eq:energy} and \eqref{eq:gamma}, respectively.
\end{prop}
\begin{proof}
Fix $u\in H^1(a,b)$ satisfying \eqref{eq:u0continI}, $v$ as in \eqref{vstruct} satisfying \eqref{eq:u-v} and fix $A\in(0,r\sqrt{2\lambda})$. 
For $\rho_1>0$ such that $[\alpha-\rho_1,\beta+\rho_1]\subset I$, define 
\begin{equation}\label{eq:nu}
	\begin{aligned}
		\omega:=&\min\left\{\frac{2G''(\alpha)-\nu\rho_1}{\max_{|z-\alpha|\leq\rho_1}D^2(z)},\frac{2G''(\beta)-\nu\rho_1}{\max_{|z-\beta|\leq\rho_1}D^2(z)}\right\},\\
		\nu&:=\sup\left\{|G'''(x)|, \,\, x\in[\alpha-\rho_1,\beta+\rho_1]\right\}.
	\end{aligned}
\end{equation}
Notice that $\omega\to2\lambda$ as $\rho_1\to0^+$. 
Take $\hat r\in(0,r)$ and $\rho_1$ so small that 
\begin{equation}\label{eq:cond-A}
	A\leq(r-\hat r)\sqrt{\omega}, \qquad \qquad \rho_1 \max_{u\in[\alpha-\rho_1,\beta+\rho_1]}|D'(u)|\leq d,
\end{equation}
where $d>0$ is the minimum of $D$, see \eqref{eq:ass-D}. 
Then, choose $0<\rho_2 < \rho_1$  sufficiently small that
\begin{equation}\label{eq:forrho2}
\begin{aligned}
	\int_{\beta-\rho_1}^{\beta-\rho_2}\sqrt{2G(s)}D(s)\,ds&>\int_{\beta-\rho_2}^{\beta}\sqrt{2G(s)}D(s)\,ds,  \\
	\int_{\alpha+\rho_2}^{\alpha+\rho_1}\sqrt{2G(s)}D(s)\,ds&> \int_{\alpha}^{\alpha+\rho_2}\sqrt{2G(s)}D(s)\,ds.
	\end{aligned}
\end{equation}
The choices of the constants $\hat r, \rho_1,\rho_2$ will be clear later on the proof.

We focus our attention on $h_i$, one of the discontinuous points of $v$ and, to fix ideas, 
let $v(h_i-r)=\alpha$, $v(h_i+r)=\beta$, the other case being analogous.
We can choose $\delta>0$ so small in \eqref{eq:u-v} so that there exist $r_+$ and $r_-$ in $(0,\hat r)$ such that
\begin{equation}\label{2points}
	|u(h_i+r_+)-\beta|<\rho_2, \qquad \quad \mbox{ and } \qquad \quad |u(h_i-r_-)-\alpha|<\rho_2.
\end{equation}
Indeed, assume by contradiction that $|u-\beta|\geq\rho_2$ throughout $(h_i,h_i+\hat r)$; then
\begin{equation*}
	\delta\geq\|u-v\|_{{}_{L^1}}\geq\int_{h_i}^{h_i+\hat r}|u-v|\,dx\geq\hat r\rho_2,
\end{equation*}
and this leads to a contradiction if we choose $\delta\in(0,\hat r\rho_2)$.
Similarly, one can prove the existence of $r_-\in(0,\hat r)$ such that $|u(h_i-r_-)-\alpha|<\rho_2$.

Now, we consider the interval $(h_i-r,h_i+r)$ and claim that
\begin{equation}\label{eq:claim}
	\int_{h_i-r}^{h_i+r}\left\{\frac\e2[D(u)u_x]^2+\frac{G(u)}\e\right\}\,dx\geq\gamma_0-\tfrac{C}N\exp(-A/\varepsilon),
\end{equation}
for some $C>0$ independent on $\e$.
Observe that from Young inequality, it follows that for any $a\leq c<d\leq b$,
\begin{equation}\label{eq:ineq}
	\int_c^d\left\{\frac\e2[D(u)u_x]^2+\frac{G(u)}\e\right\}\,dx \geq \left|\int_{u(c)}^{u(d)}\sqrt{2G(s)}D(s)\,ds\right|.
\end{equation}
Hence, if $u(h_i+r_+)\geq\beta$ and $u(h_i-r_-)\leq\alpha$, then from \eqref{eq:ineq} we can conclude that
\begin{equation*}
	\int_{h_i-r_-}^{h_i+r_+}\left\{\frac\e2[D(u)u_x]^2+\frac{G(u)}\e\right\}\,dx\geq\gamma_0,
\end{equation*}
which implies \eqref{eq:claim}.
On the other hand, notice that in general we have
\begin{align}
	\int_{h_i-r}^{h_i+r}\left\{\frac\e2[D(u)u_x]^2+\frac{G(u)}\e\right\}\,dx & 
	\geq \int_{h_i+r_+}^{h_i+r}\left\{\frac\e2[D(u)u_x]^2+\frac{G(u)}\e\right\}\,dx\notag \\ 
	& \quad + \int_{h_i-r}^{h_i-r_-}\left\{\frac\e2[D(u)u_x]^2+\frac{G(u)}\e\right\}\,dx \notag \\
	& \quad +\int_{\alpha}^{\beta}\sqrt{2G(s)}D(s)\,ds\notag\\
	&\quad-\int_{\alpha}^{u(h_i-r_-)}\sqrt{2G(s)}D(s)\,ds \notag \\
	& \quad-\int_{u(h_i+r_+)}^{\beta}\sqrt{2G(s)}D(s)\,ds\notag \\
	&=:I_1+I_2+\gamma_0-I_3-I_4, \label{eq:Pe}
\end{align}
where we again used \eqref{eq:ineq}. 
Regarding $I_1$, assume that $\beta-\rho_2<u(h_i+r_+)<\beta$ and consider the unique minimizer $z:[h_i+r_+,h_i+r]\rightarrow\R$ 
of $I_1$ subject to the boundary condition $z(h_i+r_+)=u(h_i+r_+)$.
If the range of $z$ is not contained in the interval $(\beta-\rho_1,\beta+\rho_1)$, then from \eqref{eq:ineq}, it follows that
\begin{equation}\label{E>fi}
	\int_{h_i+r_+}^{h_i+r}\left\{\frac\e2[D(u)u_x]^2+\frac{G(u)}\e\right\}\,dx>\int_{u(h_i+r_+)}^{\beta}\sqrt{2G(s)}D(s)\,ds=I_4,
\end{equation}
by the choice of $r_+$ and $\rho_2$. 
Suppose, on the other hand, that the range of $z$ is contained in the interval $(\beta-\rho_1,\beta+\rho_1)$. 
Then, the Euler-Lagrange equation for $z$ is
\begin{align*}
	&\e D^2(z)z''=\e^{-1}G'(z)-\e D(z)D'(z)(z')^2, \quad \qquad x\in(h_i+r_+,h_i+r),\\
	&z(h_i+r_+)=u(h_i+r_+), \quad \qquad z'(h_i+r)=0.
\end{align*}
Denoting by $\psi(x):=(z(x)-\beta)^2$, we have $\psi'=2(z-\beta)z'$ and 
\begin{equation*}
	\psi''=2(z-\beta)z''+2(z')^2=\frac{2G'(z)}{\varepsilon^2 D^2(z)}(z-\beta)+2\left[\frac{D(z)-(z-\beta)D'(z)}{D(z)}\right](z')^2.
\end{equation*}
Since $|z(x)-\beta|\leq\rho_1$ for any $x\in[h_i+r_+,h_i+r]$, using Taylor's expansion and the assumptions \eqref{eq:ass-G} on $G$, we get
\begin{equation*}
	G'(z(x))=G''(\beta)(z(x)-\beta)+R,
\end{equation*}
where $|R|\leq\nu|z-\beta|^2/2$ with $\nu$ defined in \eqref{eq:nu}, and as a consequence
\begin{equation*}
	\psi''(x)\geq \frac{2G''(\beta)}{\varepsilon^2D^2(z)}(z(x)-\beta)^2-\frac{\nu\rho_1}{\varepsilon^2 D^2(z)}(z(x)-\beta)^2\geq\frac{\omega}{\e^2}\psi(x) \geq\frac{\mu^2}{\varepsilon^2}\psi(x),
\end{equation*}
where $\mu:=A/(r-\hat r)$ and we used \eqref{eq:nu}-\eqref{eq:cond-A}. 
Thus, $\psi$ satisfies
\begin{align*}
	\psi''(x)-\frac{\mu^2}{\varepsilon^2}\psi(x)\geq0, \quad \qquad x\in(h_i+r_+,h_i+r),\\
	\psi(h_i+r_+)=(u(h_i+r_+)-\beta)^2, \quad \qquad \psi'(h_i+r)=0.
\end{align*}
We compare $\psi$ with the solution $\hat \psi$ of
\begin{align*}
	\hat\psi''(x)-\frac{\mu^2}{\varepsilon^2}\hat\psi(x)=0, \quad \qquad x\in(h_i+r_+,h_i+r),\\
	\hat\psi(h_i+r_+)=(u(h_i+r_+)-\beta)^2, \quad \qquad \hat\psi'(h_i+r)=0,
\end{align*}
which can be explicitly calculated to be
\begin{equation*}
	\hat\psi(x)=\frac{(u(h_i+r_+)-\beta)^2}{\cosh\left[\frac\mu\varepsilon(r-r_+)\right]}\cosh\left[\frac\mu\varepsilon(x-(h_i+r))\right].
\end{equation*}
By the maximum principle, $\psi(x)\leq\hat\psi(x)$ so, in particular,
\begin{equation*}
	\psi(h_i+r)\leq\frac{(u(h_i+r_+)-\beta)^2}{\cosh\left[\frac\mu\varepsilon(r-r_+)\right]}\leq2\exp(-A/\varepsilon)(u(h_i+r_+)-\beta)^2.
\end{equation*}
Then, we have 
\begin{equation}\label{|z-v+|<exp}
	|z(h_i+r)-\beta|\leq\sqrt2\exp(-A/2\varepsilon)\rho_2.
\end{equation}
Now, by using Taylor's expansion for $G(s)$, we obtain
\begin{equation*}
	G(s)\leq(s-\beta)^2\left(\frac{G''(\beta)}2+\frac{o(|s-\beta|^2)}{|s-\beta|^2}\right).
\end{equation*}
Therefore, for $s$ sufficiently close to $\beta$ we have
\begin{equation}\label{W-quadratic}
	0\leq G(s)\leq\Lambda(s-\beta)^2.
\end{equation}
Using \eqref{|z-v+|<exp} and \eqref{W-quadratic}, we obtain
\begin{align}
	\left|\int_{z(h_i+r)}^{\beta}\sqrt{2G(s)}D(s)\,ds\right|&\leq\sqrt{\Lambda/2}(z(h_i+r)-\beta)^2\notag \\
	&\leq\sqrt{2\Lambda}\,\rho_2^2\,\exp(-A/\varepsilon). \label{fi<exp}
\end{align}
From \eqref{eq:ineq}-\eqref{fi<exp} it follows that, for some constant $C>0$, 
\begin{align}
	\int_{h_i+r_+}^{h_i+r}\left\{\frac\e2[D(z)z_x]^2+\frac{G(z)}\e\right\}\,dx &\geq \left|\int_{z(h_i+r_+)}^{\beta}\sqrt{2G(s)}D(s)\,ds\,-\right.\nonumber \\
	&\qquad \qquad\left.\int_{z(h_i+r)}^{\beta}\sqrt{2G(s)}D(s)\,ds\right| \nonumber\\
	& \geq I_4-\tfrac{C}{2N}\exp(-A/\varepsilon). \label{E>fi-exp}
\end{align}
Combining \eqref{E>fi} and \eqref{E>fi-exp}, we get that the constrained minimizer $z$ of the proposed variational problem satisfies
\begin{equation*}	
	\int_{h_i+r_+}^{h_i+r}\left\{\frac\e2[D(z)z_x]^2+\frac{G(z)}\e\right\}\,dx \geq I_4-\tfrac{C}{2N}\exp(-A/\varepsilon).
\end{equation*}
The restriction of $u$ to $[h_i+r_+,h_i+r]$ is an admissible function, so it must satisfy the same estimate and we have
\begin{equation}\label{eq:I1}
	I_1\geq I_4-\tfrac{C}{2N}\exp(-A/\varepsilon).
\end{equation}
The term $I_2$ on the right hand side of \eqref{eq:Pe} is estimated similarly by analyzing  the interval $[h_i-r,h_i-r_-]$ 
and using the second condition of \eqref{eq:forrho2} to obtain the corresponding inequality \eqref{E>fi}.
The obtained lower bound reads:
\begin{equation}\label{eq:I2}	
	I_2\geq I_3-\tfrac{C}{2N}\exp(-A/\varepsilon).
\end{equation}
Finally, by substituting \eqref{eq:I1} and \eqref{eq:I2} in \eqref{eq:Pe}, we deduce \eqref{eq:claim}.
Summing up all of these estimates for $i=1, \dots, N$, namely for all transition points, we end up with
\begin{equation*}
	E_\varepsilon[u]\geq\sum_{i=1}^N\int_{h_i-r}^{h_i+r}\left\{\frac\e2[D(u)u_x]^2+\frac{G(u)}\e\right\}\,dx\geq N\gamma_0-C\exp(-A/\varepsilon),
\end{equation*}
and the proof is complete.
\end{proof}
Thanks to the generalized effective energy of Ginzburg-Landau type \eqref{eq:energy}, 
the dissipative estimate \eqref{eq:energy-var} and the lower bound in Proposition \ref{prop:lower}, 
we can apply the energy approach introduced in \cite{BrKo90} and we can proceed as in \cite{Grnt95}, \cite{Fol19}, \cite{FLM19}.
\begin{prop}\label{prop:L2-norm}
Assume that $f,D\in C^2(I)$ satisfy \eqref{eq:ass-D}-\eqref{eq:ass-f}-\eqref{eq:ass-int0}-\eqref{eq:ass-int1}, 
and consider the solution $u^\e$ to \eqref{eq:D-model}-\eqref{eq:Neu}-\eqref{eq:initial} 
with initial datum $u_0^{\varepsilon}$ satisfying \eqref{eq:u0continI}, \eqref{eq:ass-u0} and \eqref{eq:energy-ini}.
Then, there exist positive constants $\varepsilon_0, C_1, C_2>0$ (independent on $\varepsilon$) such that
\begin{equation}\label{L2-norm}
	\int_0^{C_1\varepsilon^{-1}\exp(A/\varepsilon)}\|u_t^\varepsilon\|^2_{{}_{L^2}}dt\leq C_2\varepsilon\exp(-A/\varepsilon),
\end{equation}
for all $\varepsilon\in(0,\varepsilon_0)$.
\end{prop}

\begin{proof}
Let $\varepsilon_0>0$ so small that for all $\varepsilon\in(0,\varepsilon_0)$, \eqref{eq:energy-ini} holds and 
\begin{equation}\label{1/2delta}
	\|u_0^\varepsilon-v\|_{{}_{L^1}}\leq\frac12\delta,
\end{equation}
where $\delta$ is the constant of Proposition \ref{prop:lower}. 
Let $\hat T>0$; we claim that if
\begin{equation}\label{claim1}
	\int_0^{\hat T}\|u_t^\varepsilon\|_{{}_{L^1}}dt\leq\frac12\delta,
\end{equation}
then there exists $C>0$ such that
\begin{equation}\label{claim2}
	E_\varepsilon[u^\varepsilon](\hat T)\geq N\gamma_0-C\exp(-A/\varepsilon).
\end{equation}
Indeed, inequality \eqref{claim2} follows from Proposition \ref{prop:lower} if $\|u^\varepsilon(\cdot,\hat T)-v\|_{{}_{L^1}}\leq\delta$.
By using triangle inequality, \eqref{1/2delta} and \eqref{claim1}, we obtain
\begin{equation*}
	\|u^\varepsilon(\cdot,\hat T)-v\|_{{}_{L^1}}\leq\|u^\varepsilon(\cdot,\hat T)-u_0^\varepsilon\|_{{}_{L^1}}+\|u_0^\varepsilon-v\|_{{}_{L^1}}
	\leq\int_0^{\hat T}\|u_t^\varepsilon\|_{{}_{L^1}}+\frac12\delta\leq\delta.
\end{equation*}
By using the inequalities \eqref{eq:energy-var}, \eqref{eq:energy-ini} and \eqref{claim2}, we deduce
\begin{equation}\label{L2-norm-Teps}
	\int_0^{\hat T}\|u_t^\varepsilon\|^2_{{}_{L^2}}dt\leq\frac{\e}{d}\left(E_\e[u_0^\e]-E_\e[u_\e](\hat T)\right)\leq C_2\e\exp(-A/\varepsilon).
\end{equation}
It remains to prove that inequality \eqref{claim1} holds for $\hat T\geq C_1\e^{-1}\exp(A/\varepsilon)$.
If 
\begin{equation*}
	\int_0^{+\infty}\|u_t^\varepsilon\|_{{}_{L^1}}dt\leq\frac12\delta,
\end{equation*}
there is nothing to prove. 
Otherwise, choose $\hat T$ such that
\begin{equation*}
	\int_0^{\hat T}\|u_t^\varepsilon\|_{{}_{L^1}}dt=\frac12\delta.
\end{equation*}
Using H\"older's inequality and \eqref{L2-norm-Teps}, we infer
\begin{equation*}
	\frac12\delta\leq[\hat T(b-a)]^{1/2}\biggl(\int_0^{\hat T}\|u_t^\varepsilon\|^2_{{}_{L^2}}dt\biggr)^{1/2}\leq
	\left[\hat T(b-a)C_2\varepsilon\exp(-A/\varepsilon)\right]^{1/2}.
\end{equation*}
It follows that there exists $C_1>0$ such that
\begin{equation*}
	\hat T\geq C_1\varepsilon^{-1}\exp(A/\varepsilon),
\end{equation*}
and the proof is complete.
\end{proof}

Now, we have all the tools to prove \eqref{eq:limit}.
\begin{proof}[Proof of Theorem \ref{thm:main}]
Triangle inequality gives
\begin{equation}\label{trianglebar}
	\|u^\varepsilon(\cdot,t)-v\|_{{}_{L^1}}\leq\|u^\varepsilon(\cdot,t)-u_0^\varepsilon\|_{{}_{L^1}}+\|u_0^\varepsilon-v\|_{{}_{L^1}},
\end{equation}
for all $t\in[0,m\exp(A/\varepsilon)]$. 
The last term of inequality \eqref{trianglebar} tends to $0$ by assumption \eqref{eq:ass-u0}.
Regarding the first term, take $\varepsilon$ so small that $C_1\varepsilon^{-1}\geq m$;
thus we can apply Proposition \ref{prop:L2-norm} and by using H\"older's inequality and \eqref{L2-norm}, we infer
\begin{equation*}
	\sup_{0\leq t\leq m\exp(A/\varepsilon)}\|u^\e(\cdot,t)-u^\e_0\|_{{}_{L^1}}\leq\int_0^{m\exp(A/\varepsilon)}\|u_t^\e(\cdot,t)\|_{{}_{L^1}}\,dt\leq C\sqrt\e,
\end{equation*}		
for all $t\in[0,m\exp(A/\varepsilon)]$. Hence \eqref{eq:limit} follows.
\end{proof}

\section{Metastable patterns and speed of the layers}\label{sec:exist}
The goal of this section is to construct a family of functions $u^\e$ having a $N$-transitions layer structure (existence of metastable patterns) 
and to give an estimate on the velocity of the transition points $h_1,\ldots,h_N$;
more precisely, we will show that the layers move with an exponentially small speed as $\e\to0$.
\subsection{Existence of metastable patterns}
In order to construct a family of functions $u^\e$ satisfying \eqref{eq:ass-u0}-\eqref{eq:energy-ini}, 
we will use a standing wave solution to \eqref{eq:D-model}, that is the solution $\Phi_\e=\Phi_\e(x)$ to the boundary value problem
\begin{equation}\label{eq:Fi}
	\begin{cases}
		\e^2\left(D(\Phi_\e)\Phi'_\e\right)'-f(\Phi_\e)=0,\qquad\mbox{in } (-\infty,+\infty),\\
		\displaystyle\lim_{x\to-\infty}\Phi_\e(x)=\alpha, \qquad\lim_{x\to+\infty}\Phi_\e(x)=\beta,\\
		\displaystyle\Phi_\e(0)=\frac{\alpha+\beta}{2},
	\end{cases}
\end{equation}
where $D,f$ satisfy \eqref{eq:ass-D}-\eqref{eq:ass-f}-\eqref{eq:ass-int0}-\eqref{eq:ass-int1}.
\begin{prop}\label{prop:ex-met}
Fix a piecewise function $v$ as in \eqref{vstruct} and assume that $f,D\in C^2(I)$ satisfy \eqref{eq:ass-D}-\eqref{eq:ass-f}-\eqref{eq:ass-int0}-\eqref{eq:ass-int1}.
Then, there exists a function $u^\e\in H^1(a,b)$ taking values in $(\alpha,\beta)$ and satisfying \eqref{eq:ass-u0}, \eqref{eq:energy-ini} and
\begin{equation}\label{eq:lim-energy-v}
	\lim_{\e\to0^+} E_\e[u^\e]=N\gamma_0,
\end{equation}
where $\gamma_0$ is defined in \eqref{eq:gamma}.
\end{prop}
\begin{proof}
First of all, we prove that if $f,D\in C^2(I)$ satisfy \eqref{eq:ass-D}-\eqref{eq:ass-f}-\eqref{eq:ass-int0}-\eqref{eq:ass-int1},
then there exists a unique increasing solution to \eqref{eq:Fi}.
Multiplying by $D(\Phi_\e)\Phi'_\e=D(\Phi_\e(x))\Phi'_\e(x)$ the first equation of \eqref{eq:Fi}, we deduce
\begin{equation*}
	\left\{\frac{\e^2}{2}\left[D(\Phi_\e)\Phi'_\e\right]^2-G(\Phi_\e)\right\}'=0, \qquad\qquad \mbox{in }\, (-\infty,+\infty),
\end{equation*}
where $G$ is defined in \eqref{eq:G} and, from the second equation of \eqref{eq:Fi} and \eqref{eq:ass-G}, 
it follows that the profile $\Phi_\e$ satisfies 	
\begin{equation}\label{eq:FI-first}
	\begin{cases}
		\e D(\Phi_\e)\Phi'_\e=\sqrt{2G(\Phi_\e)},\\
		\Phi_\e(0)=\displaystyle\frac{\alpha+\beta}{2}.
	\end{cases}
\end{equation}
Therefore, the fact that $G$ satisfies \eqref{eq:ass-G} and the strictly positiveness of $D$ imply that 
there exists a unique solution to \eqref{eq:FI-first} which is increasing and implicitly defined by
\begin{equation}\label{eq:Fi-implicit}
	\int_{\frac{\alpha+\beta}{2}}^{\Phi_\e(x)}\displaystyle\frac{D(s)}{\sqrt{2G(s)}}\,ds=\frac{x}{\e}.
\end{equation} 
Observe that 
\begin{equation*}
	\lim_{\e\to0}\Phi_\e(x)=\begin{cases}
	\alpha, \qquad & x<0,\\	
	\frac{\alpha+\beta}{2}, &x=0,\\
	\beta, & x>0.
	\end{cases}
\end{equation*}
Now, we use the profile $\Phi_\e$ to construct a family of functions satisfying \eqref{eq:ass-u0} and \eqref{eq:energy-ini}.
Fix $N\in\mathbb{N}$ and $N$ transition points $a<h_1<h_2<\dots<h_N<b$, and denote the middle points by
\begin{equation*}
	m_1:=a, \qquad \quad m_j:=\frac{h_{j-1}+h_j}{2}, \quad j=2,\dots,N, \qquad \quad m_{N+1}:=b.
\end{equation*}
Define
\begin{equation}\label{eq:translayer}
	u^\e(x):=\Phi_\e\left((-1)^j(x-h_j)\right),  \qquad x\in[m_j,m_{j+1}],  \qquad j=1,\dots N,
\end{equation}
where $\Phi_\e$ is the solution to \eqref{eq:Fi}.
Notice that $u^\e(h_j)=\frac{\alpha+\beta}{2}$, for $j=1,\dots,N$ and for definiteness we choose $u^\e(a)<0$ (the case $u^\e(a)>0$ is analogous).
It is easy to check that $u^\e\in H^1(a,b)$, $\alpha<u^\e<\beta$ and \eqref{eq:ass-u0} holds; 
let us prove that $u^\e$ satisfies \eqref{eq:energy-ini}.
From the definitions of $E_\e$ \eqref{eq:energy}, $u^\e$ \eqref{eq:translayer} and \eqref{eq:FI-first}, we obtain
\begin{align*}
	E_\e[u^\e]&=\sum_{j=1}^{N}\int_{m_j}^{m_{j+1}}\left[\frac\e2[D(u^\e)u^\e_x]^2+\frac{G(u^\e)}\e\right]\,dx=\sum_{j=1}^{N}\int_{m_j}^{m_{j+1}}\frac{2G(\Phi_\e)}\e\,dx\\
	&=\sum_{j=1}^{N}\int_{m_j}^{m_{j+1}}\sqrt{2G(\Phi_\e)}D(\Phi_\e)|\Phi'_\e|\,dx<N\gamma_0,
	\end{align*}
where $\gamma_0$ is defined in \eqref{eq:gamma}, and then $u^\e$ satisfies \eqref{eq:energy-ini}.
Moreover, choosing $\e$ so small that assumption \eqref{eq:u-v} is satisfied and applying Proposition \ref{prop:lower}, we infer
\begin{equation*}
	N\gamma_0-C\exp(-A/\varepsilon)\leq E_\e[u^\e]<N\gamma_0.
\end{equation*}
Passing to the limit as $\e\to0^+$, we end up with \eqref{eq:lim-energy-v} and the proof is complete.
\end{proof}
In what follows, we discuss the role of assumptions \eqref{eq:ass-D}-\eqref{eq:ass-f} on the functions $f,D$ in the proof of Proposition \ref{prop:ex-met}.
In particular, we discuss their role in \eqref{eq:FI-first} and we see what happens when they are not satisfied, that is when either $f'(\alpha)f'(\beta)=0$ or $D(\alpha)D(\beta)=0$.

First of all, notice that the assumptions \eqref{eq:ass-D}, \eqref{eq:ass-f} imply the exponential decay
\begin{equation}\label{eq:exp-decay}
	\begin{aligned}
		&|\Phi_\e(x)-\alpha|\leq c_1e^{c_2 x}, \qquad\qquad&\mbox{ as }\, x\to-\infty,\\
		&|\Phi_\e(x)-\beta|\leq c_1e^{-c_2 x}, \qquad\qquad&\mbox{ as }\, x\to+\infty,
	\end{aligned}
\end{equation}
for some constants $c_1,c_2>0$ (depending on $f$ and $D$). 
If $D$ is strictly positive, but $f'$ is \emph{degenerate} at $\alpha$ or $\beta$, that is $f'(\alpha)f'(\beta)=0$, 
we have the existence of a unique solution for \eqref{eq:Fi}, but we do not have the exponential decay \eqref{eq:exp-decay}.
We will see in Section \ref{sec:num} that if $f'(\alpha)f'(\beta)=0$, the numerical solutions do not exhibit exponentially slow motion, cfr. Figure \ref{fig:f-deg}.

On the other hand, in the case assumptions \eqref{eq:ass-f} are satisfied, but $D$ is \emph{degenerate} at $\alpha$ or $\beta$, that is, 
$D(\alpha)D(\beta)=0$ the situation drastically changes.
Indeed, for definiteness assume $D(\beta)=0$ and $D(s)\sim(s-\beta)^{2n}$ as $(s-\beta)\to0$, for some $n\in\mathbb{N}$. 
Using the expansion $G(s)\sim(s-\beta)^{2n+2}$, we deduce that the integral
\begin{equation}\label{eq:sharp-prof}
	\int_{\frac{\alpha+\beta}{2}}^\beta\frac{D(s)}{\sqrt{2G(s)}}\,ds<+\infty,
\end{equation}
and as a consequence, there exists $\bar x>0$ such that $\Phi_\e(\bar x)=\beta$.
In Section \ref{sec:num}, we consider a numerical solution in the case \eqref{eq:sharp-prof}
and we observe a \emph{sharp} connection between the two stable points,
which does not exhibit exponentially slow motion, cfr. Figure \ref{fig:D-deg}.

Finally, we remark that in order to have a strictly monotone solution to \eqref{eq:Fi} in the case $D(s)\sim(s-\beta)^{2n}$ as $(s-\beta)\to0$,
the potential $G$ must satisfy $G(s)\sim(s-\beta)^{2m}$ as $(s-\beta)\to0$, for some $m\geq2n+1$,
meaning that $f(\beta)=f'(\beta)=\cdots=f^{(2m-2n-2)}(\beta)=0$ and $f^{(2m-2n-1)}(\beta)>0$.
In particular, we have the exponential decay \eqref{eq:exp-decay} if and only if $m=2n+1$, 
that is $D(s)\sim(s-\beta)^{2n}$ and $f(s)\sim(s-\beta)^{2n+1}$ as  $(s-\beta)\to0$.

\subsection{Layers speed}
Theorem \ref{thm:main} and Proposition \ref{prop:ex-met} provide existence and persistence of metastable states 
with a $N$-\emph{transition layer structure} for the IBVP \eqref{eq:D-model}-\eqref{eq:Neu}-\eqref{eq:initial}.
The next goal is to establish an upper bound on the velocity of the transition points.
To do this, fix $v$ as in \eqref{vstruct} and define its {\it interface} $I[v]$ as
\begin{equation*}
	I[v]:=\{h_1,h_2,\ldots,h_N\}.
\end{equation*}
For an arbitrary function $u:[a,b]\rightarrow I$ and an arbitrary closed subset $K\subset I\backslash\{\alpha,\beta\}$,
the {\it interface} $I_K[u]$ is defined by
\begin{equation*}
	I_K[u]:=u^{-1}(K).
\end{equation*}
Finally, we recall that for any $X,Y\subset\mathbb{R}$ the {\it Hausdorff distance} $d(X,Y)$ between $X$ and $Y$ is defined by 
\begin{equation*}
	d(X,Y):=\max\biggl\{\sup_{x\in X}d(x,Y),\,\sup_{y\in Y}d(y,X)\biggr\},
\end{equation*}
where $d(x,Y):=\inf\{|y-x|: y\in Y\}$. 

The following result is purely variational in character and states that, if a function $u\in H^1(a,b)$  satisfies \eqref{eq:u0continI}, 
it is close to $v$ in $L^1$ and $E_\varepsilon[u]$ exceeds of a small quantity the minimum energy to have $N$ transitions, then 
the distance between the interfaces $I_K[u]$ and $I_K[v]$ is small.  
\begin{lem}\label{lem:interface}
Assume that $D\in C^1(I)$ satisfies \eqref{eq:ass-D} and that $G\in C^3(I)$ satisfies \eqref{eq:ass-G}. 
Given $\delta_1\in(0,r)$ and a closed subset $K\subset I\backslash\{\alpha,\beta\}$, 
there exist positive constants $\hat\delta,\varepsilon_0$ (independent on $\e$) and $M>0$ such that for any $u\in H^1(a,b)$ satisfying \eqref{eq:u0continI} and
\begin{equation}\label{eq:lem-interf}
	\|u-v\|_{{}_{L^1}}<\hat\delta \qquad \quad \mbox{ and } \qquad \quad E_\varepsilon[u]\leq N\gamma_0+M,
\end{equation}
for all $\varepsilon\in(0,\varepsilon_0)$, we have
\begin{equation}\label{lem:d-interfaces}
	d(I_K[u], I[v])<\tfrac12\delta_1.
\end{equation}
\end{lem}
\begin{proof}
Fix $\delta_1\in(0,r)$ and choose $\rho>0$ small enough that 
\begin{equation*}
	I_\rho:=(\alpha-\rho,\alpha+\rho)\cup(\beta-\rho,\beta+\rho)\subset I\backslash K, 
\end{equation*}
and 
\begin{equation*}
	\inf\left\{\left|\int_{\xi_1}^{\xi_2}\sqrt{2G(s)}D(s)\,ds\right| : \xi_1\in K, \xi_2\in I_\rho\right\}>2M,
\end{equation*}
where
\begin{equation*}
	M:=2N\max\left\{\int_{\alpha}^{\alpha+\rho}\sqrt{2G(s)}D(s)\,ds, \, \int_{\beta-\rho}^{\beta}\sqrt{2G(s)}D(s)\,ds \right\}.
\end{equation*}
By reasoning as in the proof of \eqref{2points} in Proposition \ref{prop:lower}, we can prove that for each $i$ there exist
\begin{equation*}
	x^-_{i}\in(h_i-\delta_1/2,h_i) \qquad \textrm{and} \qquad x^+_{i}\in(h_i,h_i+\delta_1/2),
\end{equation*}
such that
\begin{equation*}
	|u(x^-_{i})-v(x^-_{i})|<\rho \qquad \textrm{and} \qquad |u(x^+_{i})-v(x^+_{i})|<\rho.
\end{equation*}
Suppose that \eqref{lem:d-interfaces} is violated. 
Using \eqref{eq:ineq}, we deduce
\begin{align}
	E_\varepsilon[u]\geq&\sum_{i=1}^N\left|\int_{u(x^-_{i})}^{u(x^+_{i})}\sqrt{2G(s)}D(s)\,ds\right|\notag\\ 
	& \qquad +\inf\left\{\left|\int_{\xi_1}^{\xi_2}\sqrt{2G(s)}D(s)\,ds\right| : \xi_1\in K, \xi_2\in I_\rho\right\}. \label{diseq:E1}
\end{align}
On the other hand, we have
\begin{align*}
	\left|\int_{u(x^-_{i})}^{u(x^+_{i})}\sqrt{2G(s)}D(s)\,ds\right|&\geq\int_{\alpha}^{\beta}\sqrt{2G(s)}D(s)\,ds\\
	&\qquad-\int_{\alpha}^{\alpha+\rho}\sqrt{2G(s)}D(s)\,ds\\
	&\qquad -\int_{\beta-\rho}^{\beta}\sqrt{2G(s)}D(s)\,ds\\
	&\geq\gamma_0-\frac{M}{N}. 
\end{align*}
Substituting the latter bound in \eqref{diseq:E1}, we deduce
\begin{equation*}
	E_\varepsilon[u]\geq N\gamma_0-M+\inf\left\{\left|\int_{\xi_1}^{\xi_2}\sqrt{2G(s)}D(s)\,ds\right| : \xi_1\in K, \xi_2\in I_\rho\right\}.
\end{equation*}
For the choice of $\rho$,  we obtain	
\begin{align*}
	E_\varepsilon[u]>N\gamma_0+M,
\end{align*}
which is a contradiction with assumption \eqref{eq:lem-interf}. Hence, the bound \eqref{lem:d-interfaces} is true.
\end{proof}

Thank to Theorem \ref{thm:main} and Lemma \ref{lem:interface} we can prove the following result, 
which states that the velocity of the transition points is (at most) exponentially small.
\begin{thm}\label{thm:interface}
Assume that $f,D\in C^2(I)$ satisfy \eqref{eq:ass-D}-\eqref{eq:ass-f}-\eqref{eq:ass-int0}-\eqref{eq:ass-int1}.
Let $u^\varepsilon$ be the solution of \eqref{eq:D-model}-\eqref{eq:Neu}-\eqref{eq:initial}, 
with initial datum $u_0^{\varepsilon}$ satisfying \eqref{eq:u0continI}, \eqref{eq:ass-u0} and \eqref{eq:energy-ini}. 
Given $\delta_1\in(0,r)$ and a closed subset $K\subset I\backslash\{\alpha,\beta\}$, set
\begin{equation*}
	t_\varepsilon(\delta_1)=\inf\{t:\; d(I_K[u^\varepsilon(\cdot,t)],I_K[u_0^\varepsilon])>\delta_1\}.
\end{equation*}
There exists $\varepsilon_0>0$ such that if $\varepsilon\in(0,\varepsilon_0)$ then
\begin{equation*}
	t_\varepsilon(\delta_1)>\exp(A/\varepsilon).
\end{equation*}
\end{thm}
\begin{proof}
Let $\varepsilon_0>0$ so small that \eqref{eq:ass-u0}-\eqref{eq:energy-ini}
imply $u_0^\varepsilon$ satisfies \eqref{eq:lem-interf} for all $\varepsilon\in(0,\varepsilon_0)$.
From Lemma \ref{lem:interface} it follows that
\begin{equation}\label{interfaces-u0}
	d(I_K[u_0^\varepsilon], I[v])<\tfrac12\delta_1.
\end{equation}
Now, consider $u^\varepsilon(\cdot,t)$ for all $t\leq\exp(A/\varepsilon)$.
Assumption \eqref{eq:lem-interf} is satisfied thanks to \eqref{eq:limit} and because $E_\varepsilon[u^\varepsilon](t)$ is a non-increasing function of $t$.
Then,
\begin{equation}\label{interfaces-u}
	d(I_K[u^\varepsilon(t)], I[v])<\tfrac12\delta_1
\end{equation}
for all $t\in(0,\exp(A/\varepsilon))$. 
Combining \eqref{interfaces-u0} and \eqref{interfaces-u}, we obtain
\begin{equation*}
	d(I_K[u^\varepsilon(t)],I_K[u_0^\varepsilon])<\delta_1
\end{equation*}
for all $t\in(0,\exp(A/\varepsilon))$.
\end{proof}

\section{Numerical experiments}\label{sec:num}
In this section, we present some numerical solutions to the IBVP \eqref{eq:D-model}-\eqref{eq:Neu}-\eqref{eq:initial}, which confirm the analytical results of the previous sections.
Moreover, we consider as well the case when the assumptions \eqref{eq:ass-D} or \eqref{eq:ass-f} are not satisfied,
and we numerically show that they are necessary conditions for the exponentially slow motion of the solutions. 

\subsection{Classical Allen-Cahn equation}
As first example, we consider the classical Allen-Cahn equation
\begin{equation}\label{eq:AC-prototype}
	u_t=\e^2 u_{xx}+u-u^3,
\end{equation}
which corresponds to the choices $D(u) \equiv D_0 = 1$ and $f(u)=u^3-u$ in equation \eqref{eq:D-model}.
Therefore, we are considering the case of a constant diffusion coefficient (independent on the density $u$) 
and the simplest case of balanced bistable reaction term $f$, with two stable zeros at $\pm1$ and one unstable zero at $0$.
The metastable dynamics of the solutions to the IBVP associated to equation \eqref{eq:AC-prototype} has been studied in \cite{BrKo90}; 
moreover, equation \eqref{eq:AC-prototype} is the typical example of the general case considered in \cite{CaPe89,ChenX04,FuHa89}.
In particular, it is well know that metastable patterns for such a model can be approximated by using the function defined in \eqref{eq:translayer}, 
with $\Phi_\e(x):=\tanh(x/\sqrt2\e)$, which is the explicit solution of the boundary problem \eqref{eq:Fi} with $D\equiv1$, $f(u)=u^3-u$, $\alpha=-1$ and $\beta=+1$.

In Figure \ref{fig:classic}, we consider the interval $[a,b]=[-4,4]$ and an initial datum with a 6-transition layer structure, 
given by formula \eqref{eq:translayer} with $\Phi_\e(x):=\tanh(x/\sqrt2\e)$.

\begin{figure}[hbtp]
\centering
\includegraphics[scale = 0.4]{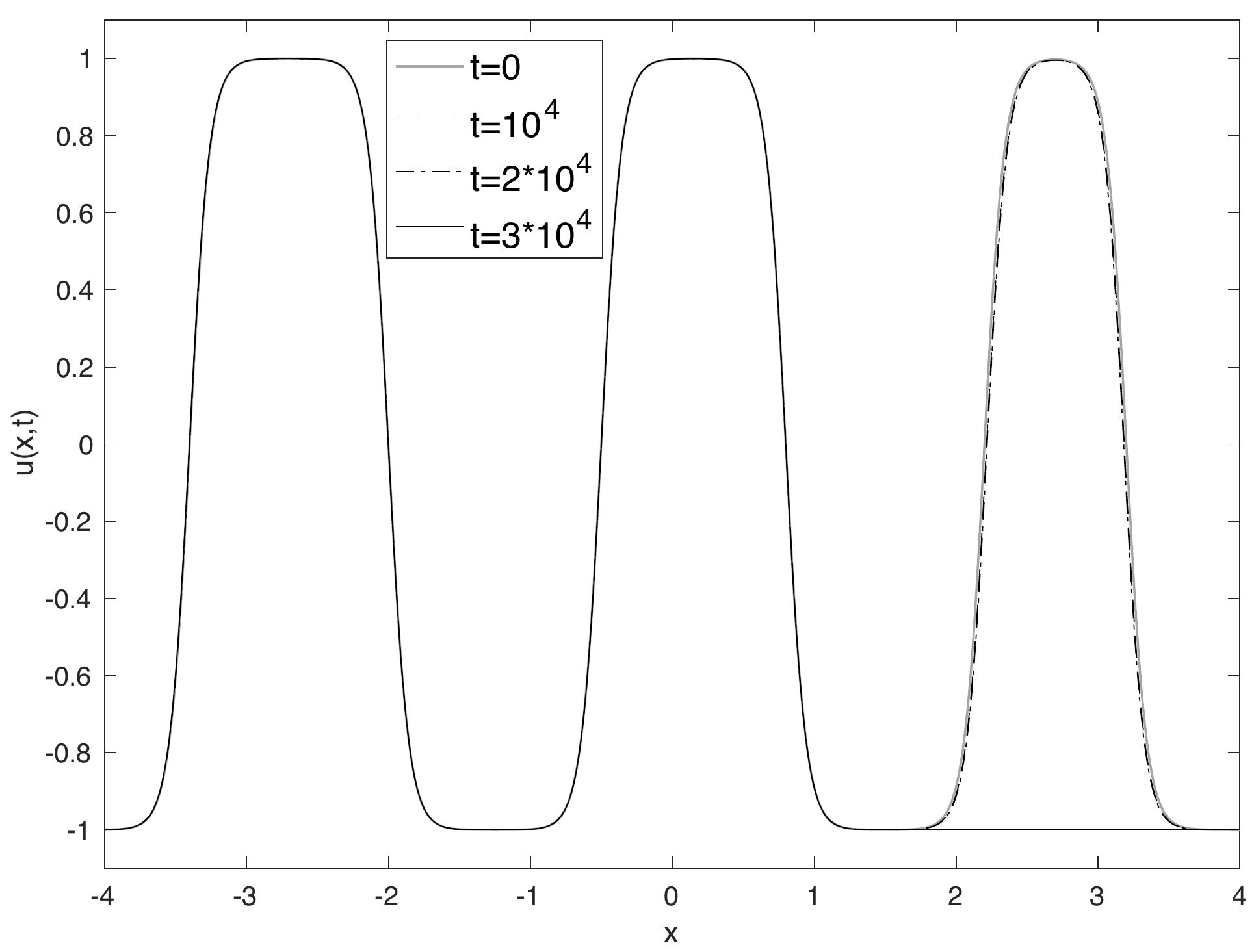}
\caption{Numerical solutions to \eqref{eq:D-model}-\eqref{eq:Neu}-\eqref{eq:initial} with $D(u)=1$, $f(u)=u(u^2-1)$, $\e=0.1$
and initial datum $u_0$ with $N = 6$ transitions, located at $-3.4, -2, -0.5, 0.8, 2.2,3.2$.}
\label{fig:classic}
\end{figure}

We see that the solution maintains the same transition layer structure of the initial datum for a time $t=2*10^4$ 
and after that the two closest transition points collapse (we choose the minimum distance equal to $1$).
Then, we have a solution with a $4$-transition layer structure and it is impossible to distinguish the solution at time $t=0$ and $t=3*10^4$ 
(apart from the elimination of the closest transition points).
Finally, we stress that for Theorem \ref{thm:main}, we expect the solution to maintain the transition layer structure for a time $t\geq\exp(A/\e)$: 
in the example considered in Figure \ref{fig:classic}, we have $A\in(0,\sqrt{2\lambda}/2)$ (since the minimum distance between the layers is $1$), 
$\lambda=2$ (see \eqref{eq:lambda}), $\e=0.1$
and so, $t\geq\exp(10)\approx2.2*10^4$.

As we already mentioned, the metastable dynamics of the solutions to equation \eqref{eq:AC-prototype} has been studied in different papers, 
but we consider it as first example in order to analyze the differences with the density dependent diffusion we consider in the next examples.

\subsection{Mullins diffusion}
Now, we consider the Mullins diffusion $D(u)=(1+u^2)^{-1}$ introduced in Section \ref{sec:intro}.
In this case, $D(u)\in(0,1]$ for any $u\in\R$ and so the diffusion coefficient is smaller than the one considered in Figure \ref{fig:classic}.
Since $D$ is an even function, we can consider the same reaction term $f$ of Figure \ref{fig:classic}, 
that is the odd function $f(u)=u^3-u$, and the function $G$ defined in \eqref{eq:G} satisfies the assumption \eqref{eq:ass-G} with $\alpha=-1$ and $\beta=1$.
In Figure \ref{fig:Mullins}, we consider the same data of Figure \ref{fig:classic} and we only change the diffusion coefficient.

\begin{figure}[hbtp]
\centering
\includegraphics[scale = 0.4]{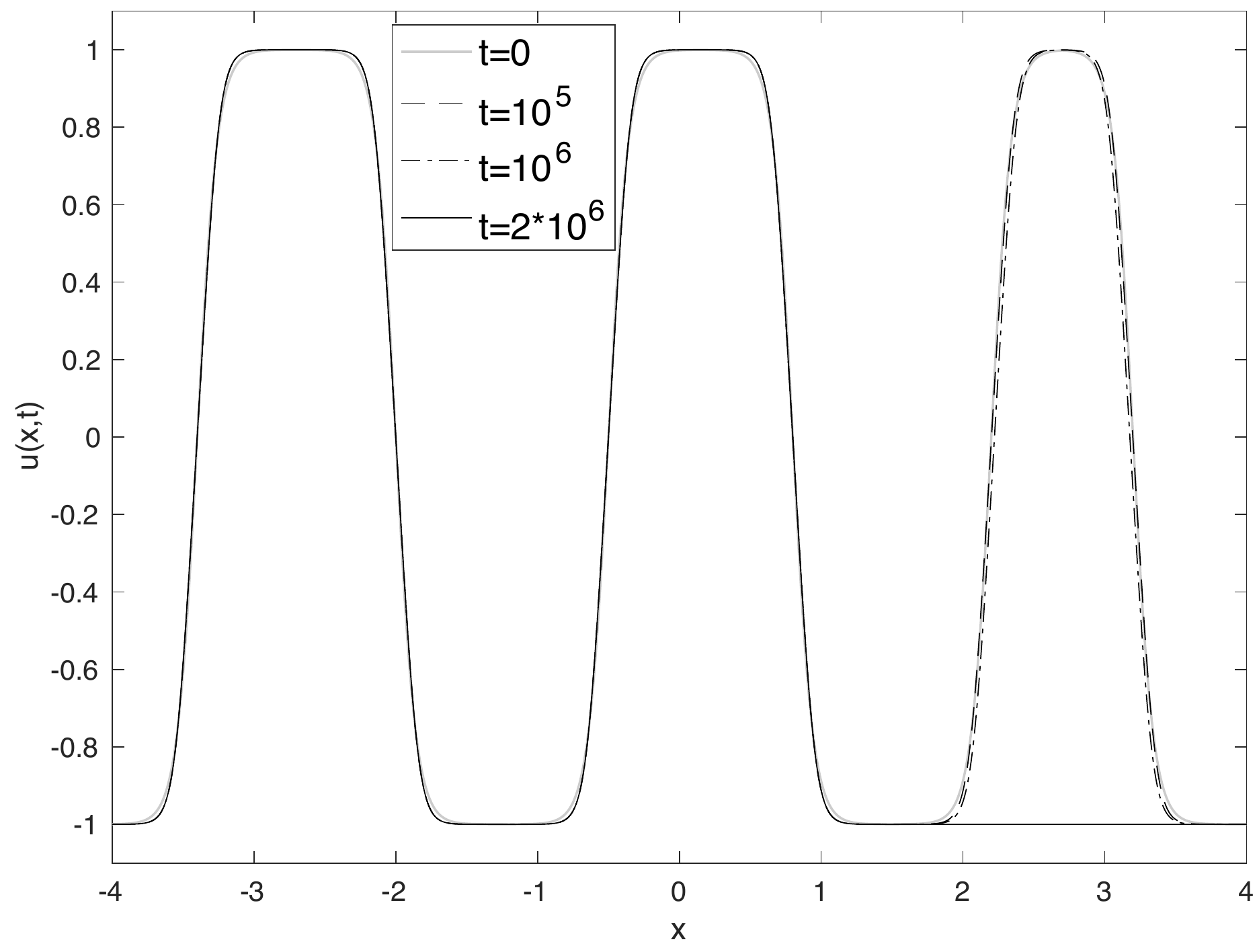}
\caption{Numerical solutions to \eqref{eq:D-model}-\eqref{eq:Neu}-\eqref{eq:initial} with $D(u)=(1+u^2)^{-1}$, $f(u)=u(u^2-1)$ and $\e=0.1$.
The initial datum $u_0$ with 6 transitions is the same of Figure \ref{fig:classic}.}
\label{fig:Mullins}
\end{figure}

In this case, we see that the lifetime of the metastable state is greater than Figure \ref{fig:classic}, 
and the solution maintains the 6-transition layer structure for a time $t=10^6$.
The numerical solution confirms the analytical results, since $D(\pm1)=1/2$ and so $\lambda=4$, with $\lambda$ defined in \eqref{eq:lambda}.
Hence, we expect the collapse of the two closest transition points after a time $t\geq\exp(10\sqrt2)\approx1.4*10^6$.

\subsection{Exponential diffusion}
Here, we consider the exponential diffusion $D(u)=e^u$, which satisfies $D(u)\in[e^{-1},e]$ for any $u\in[-1,1]$.
In particular, the diffusion coefficient is smaller than Figure \ref{fig:classic} for $u\in[-1,0)$ and it is greater for $u\in(0,1]$.
We want assumptions \eqref{eq:ass-int0}-\eqref{eq:ass-int1} to be satisfied and then we choose the reaction term of the form $f(u)=(u-u_*)(u^2-1)$;
in such a way, $\pm1$ are the stable points, while the unstable point $u_*$ has to be chosen so that $\displaystyle\int_{-1}^1f(s)e^s\,ds=0$.
It is easy to check that $u_*=\frac{e^2-7}{2}$ and it follows that $\lambda=\min\{2(1+u_*)e,2(1-u_*)e^{-1}\}=9e^{-1}-e$ and we can choose $A\in(0,0.5443)$.
In the case $\e=0.1$, we have $\exp(5.443)\approx231$ and such a number is very small (with respect to the cases considered in Figures \ref{fig:classic}-\ref{fig:Mullins});
we see in Figure \ref{fig:exp01} that the solution maintains the 6-transition layer structure for a small time.
Next, we choose $\e=0.05$ and we see in Figure \ref{fig:exp005} that the solution maintains the 6-transition layer structure for a time $t\approx10^5$.
Notice also that all the transition points move faster.

\begin{figure}[hbtp]
\centering
\subfigure[$\e=0.1$]{\label{fig:exp01}\includegraphics[scale = 0.4]{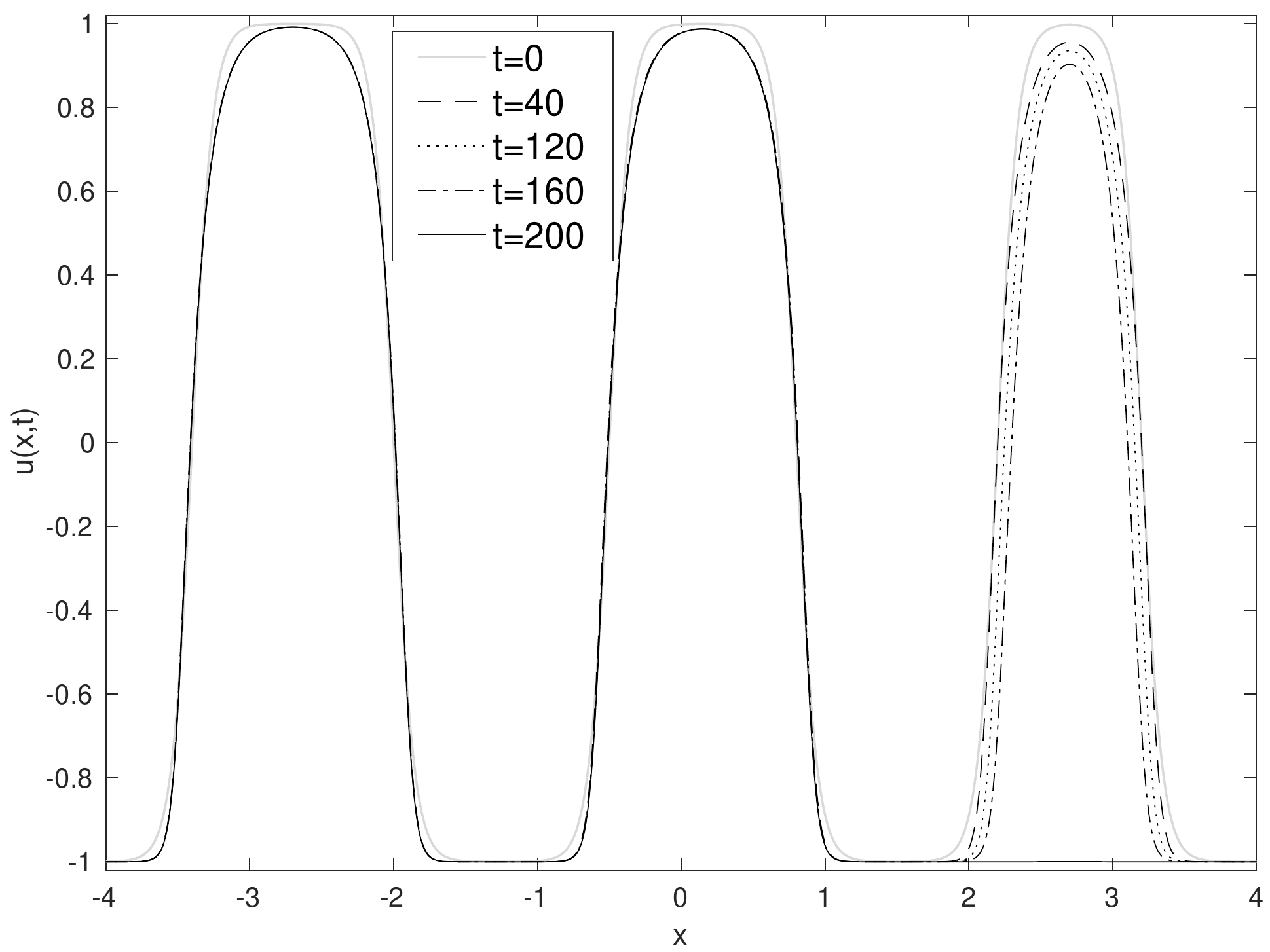}}
\subfigure[$\e=0.05$]{\label{fig:exp005}\includegraphics[scale = 0.4]{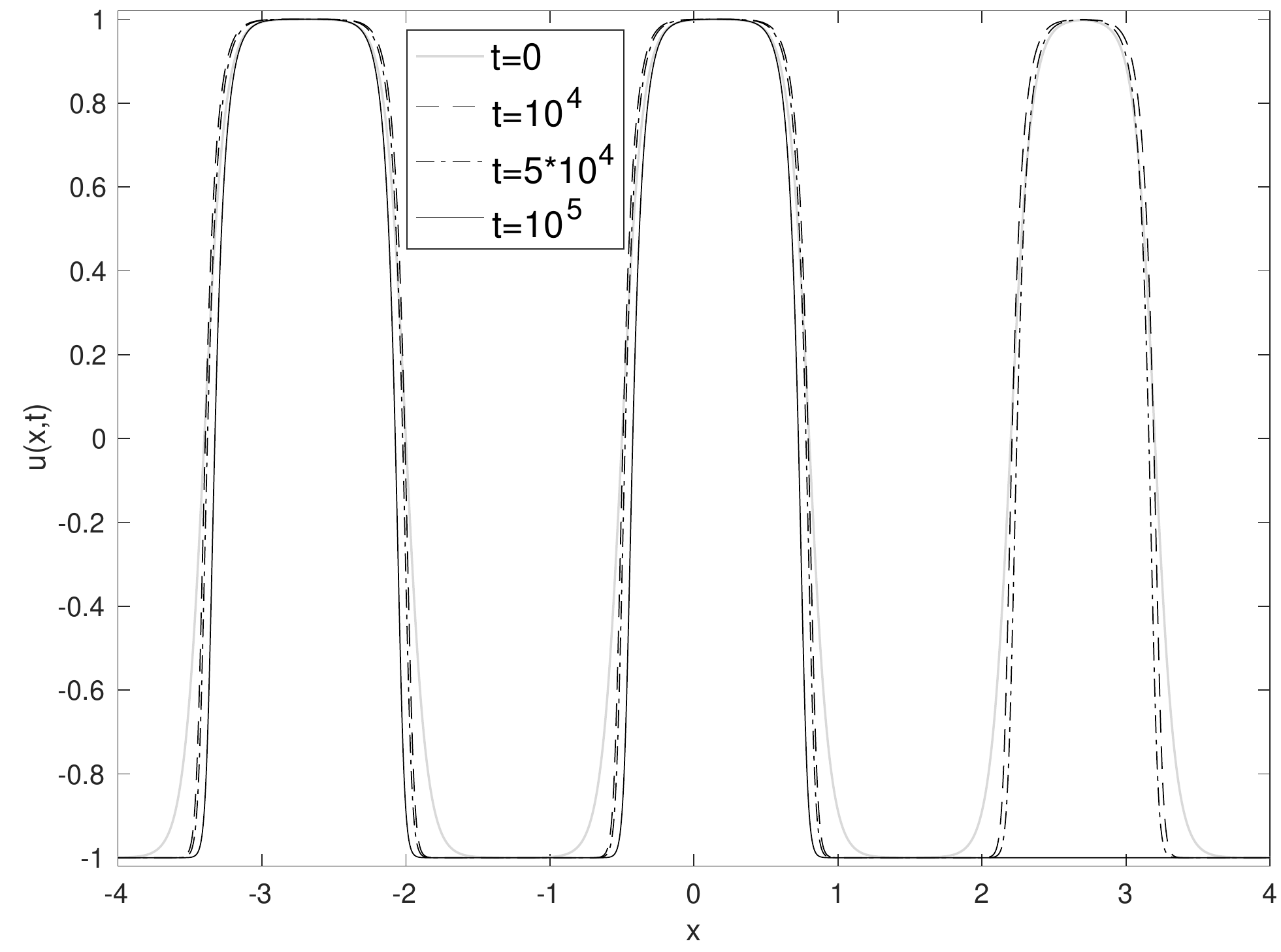}}
\caption{Numerical solutions to \eqref{eq:D-model}-\eqref{eq:Neu}-\eqref{eq:initial} with $D(u)=e^u$, $f(u)=(u-\frac{e^2-7}{2})(u^2-1)$ and two different values for $\e$.
The initial datum $u_0$ is the same of Figures \ref{fig:classic}-\ref{fig:Mullins}.}
\label{fig:exp}
\end{figure}

In all the previous example, all the assumptions of Theorem \ref{thm:main} on the nonlinear diffusion $D$ and on the reaction term $f$ are satisfied;
in the next examples we consider the \emph{degenerate cases} $f'(\alpha)=f'(\beta)=0$ or $D(\alpha)=0$.

\subsection{The degenerate case $f'(\alpha)=f'(\beta)=0$.}
Here, we consider the case when the assumptions \eqref{eq:ass-D}, \eqref{eq:ass-int0} and \eqref{eq:ass-int1} hold, $f(\alpha)=f(\beta)=0$, but $f'(\alpha)=f'(\beta)=0$.
Notice that, in such a case the function $G$ defined in \eqref{eq:G} satisfies $G''(\alpha)=G''(\beta)=0$ and so the assumptions of Proposition \ref{prop:lower} are not satisfied.
In Figure \ref{fig:f-deg}, we choose the Mullins diffusion $D(u)=(1+u^2)^{-1}$ as in Figure \ref{fig:Mullins} and we only change the reaction term.

\begin{figure}[hbtp]
\centering
\subfigure[$f(u)=u(u^2-1)^3$]{\label{fig:f-deg3}\includegraphics[scale = 0.4]{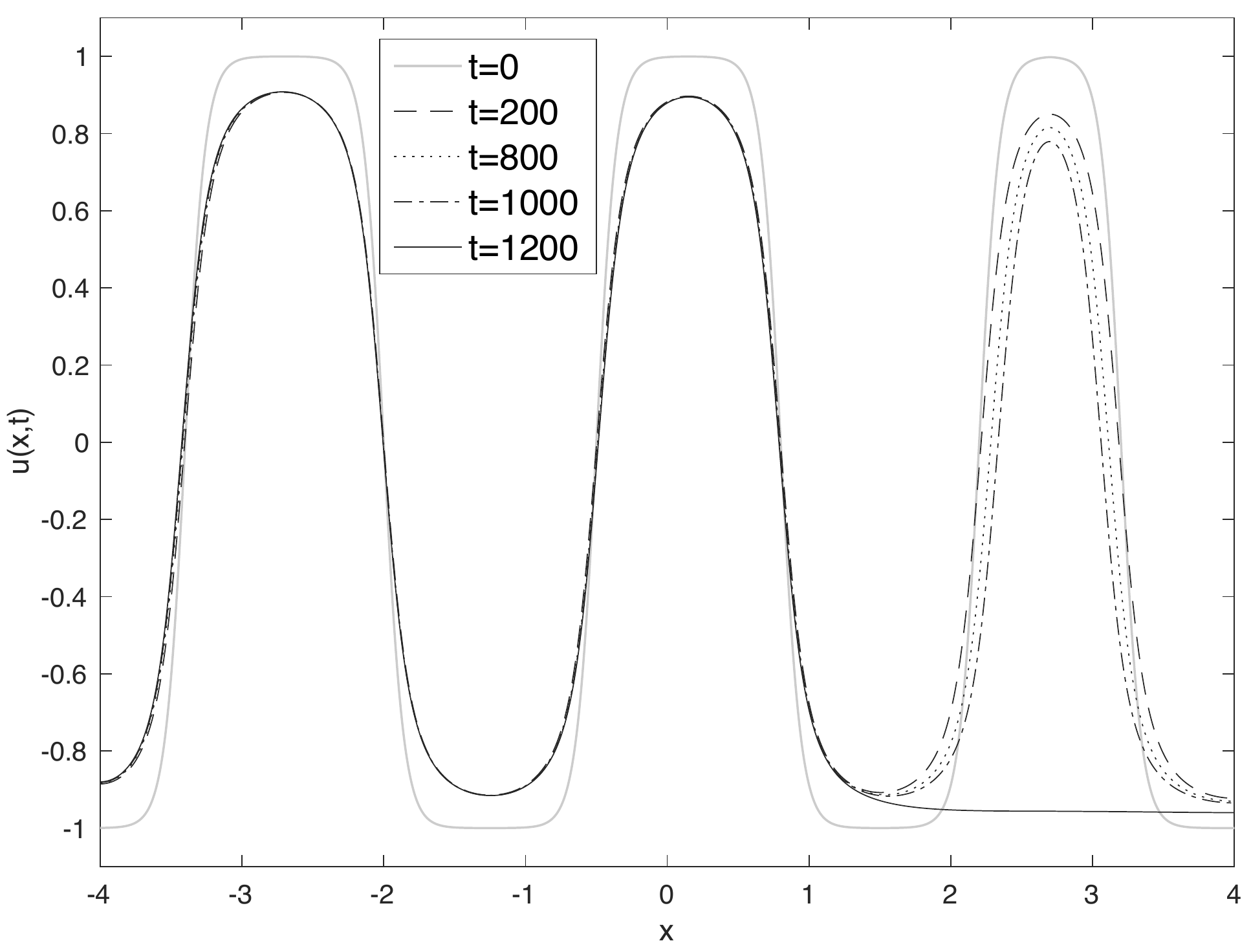}}
\subfigure[$f(u)=u(u^2-1)^5$]{\label{fig:f-deg5}\includegraphics[scale = 0.4]{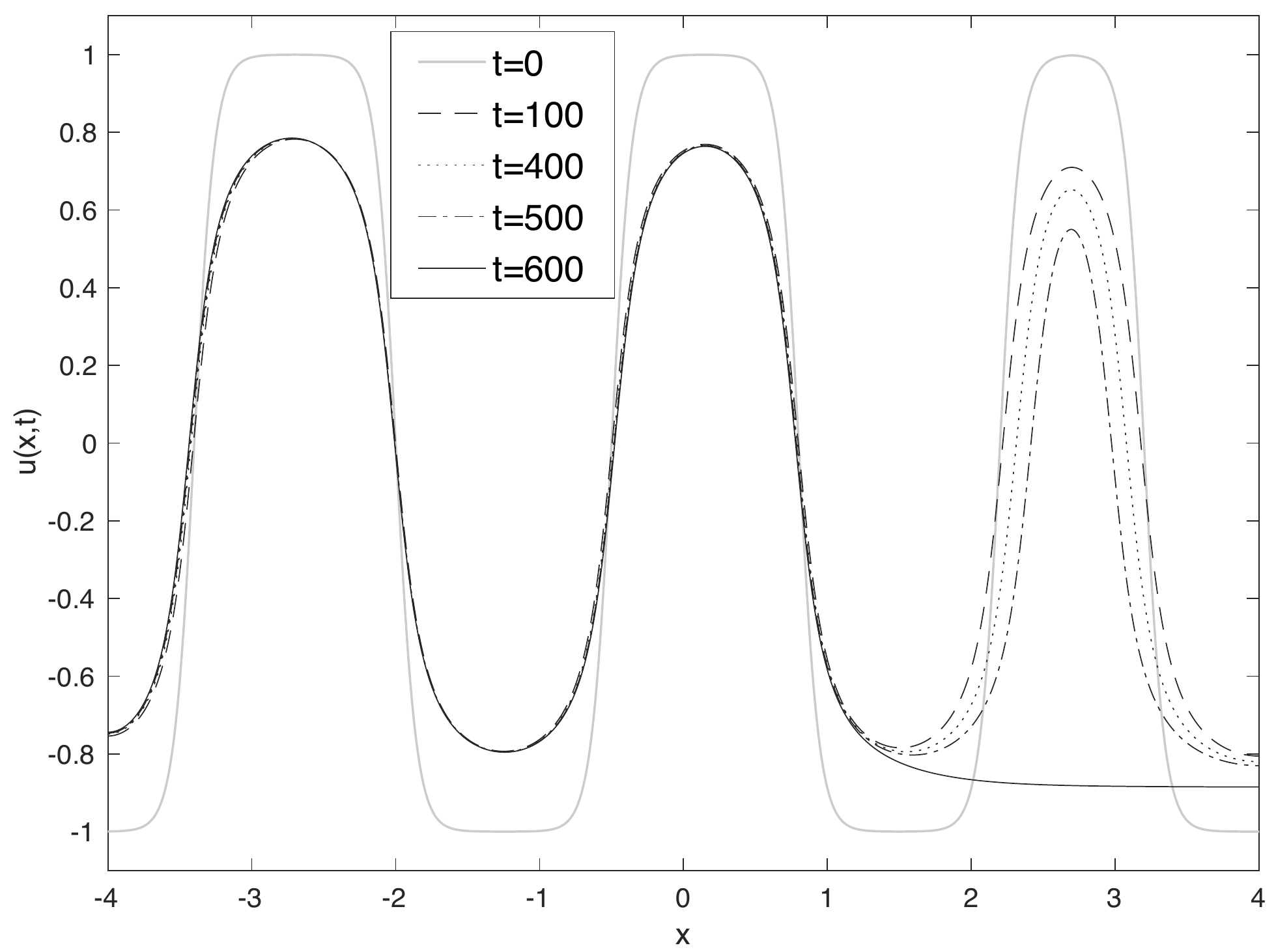}}
\caption{Numerical solutions to \eqref{eq:D-model}-\eqref{eq:Neu}-\eqref{eq:initial} with $D(u)=(1+u^2)^{-1}$, $\e=0.1$
and two different reaction terms satisfying $f'(\alpha)=f'(\beta)=0$.
The initial datum $u_0$ has $6$ transitions as in Figure \ref{fig:Mullins}.}
\label{fig:f-deg}
\end{figure}

It is very important to notice that we choose the same value of $\e$ and the same initial transition layer structure of Figure \ref{fig:Mullins},
but in the \emph{degenerate case} $f'(\alpha)=f'(\beta)=0$ the situation drastically changes and the closest layers
disappear in a much shorter time.
More precisely, in Figure \ref{fig:Mullins} with $f(u)=u(u^2-1)$ the transitions disappear after a time $t=2*10^6$,
in Figure \ref{fig:f-deg3}  with $f(u)=u(u^2-1)^3$ they disappear at time $t=1200$  
and in Figure \ref{fig:f-deg5} with $f(u)=u(u^2-1)^5$ at time $t=600$.

\subsection{The degenerate case $D(\alpha)=0$.}
In this last section, we consider an example where the diffusion coefficient vanishes in one of the two stable points $\alpha,\beta$.
Our choice follows the well-known \emph{porous medium} diffusivity \cite{Vaz07} with $D(u)=u^2$ and a reaction term satisfying \eqref{eq:ass-f} with $\alpha=0$, $\beta=1$
and such that $\displaystyle\int_0^1 f(s)s^2\,ds=0$. 
It is easy to check that the reaction term $f(u)=u(u-\frac23)(u-1)$ satisfies \eqref{eq:ass-f} and \eqref{eq:ass-int0}-\eqref{eq:ass-int1},
with $D(u)=u^2$.
As we already mentioned in the discussion after Proposition \ref{prop:ex-met}, 
in this case the stationary profile connecting the two stable points is \emph{sharp} (see \eqref{eq:sharp-prof}) 
and this is confirmed by the numerical solutions in Figure \ref{fig:D-deg}.
It is to be observed that the lifetime of the metastable state is much smaller than the previous simulations 
both in the case $\e=0.1$ (see Figure \ref{fig:deg2-eps01}) and $\e=0.06$ (see Figure \ref{fig:deg2-eps006}),
and the numerical solutions do not exhibit exponentially slow motion.

\begin{figure}[hbtp]
\centering
\subfigure[$\e=0.1$]{\label{fig:deg2-eps01}\includegraphics[scale = 0.4]{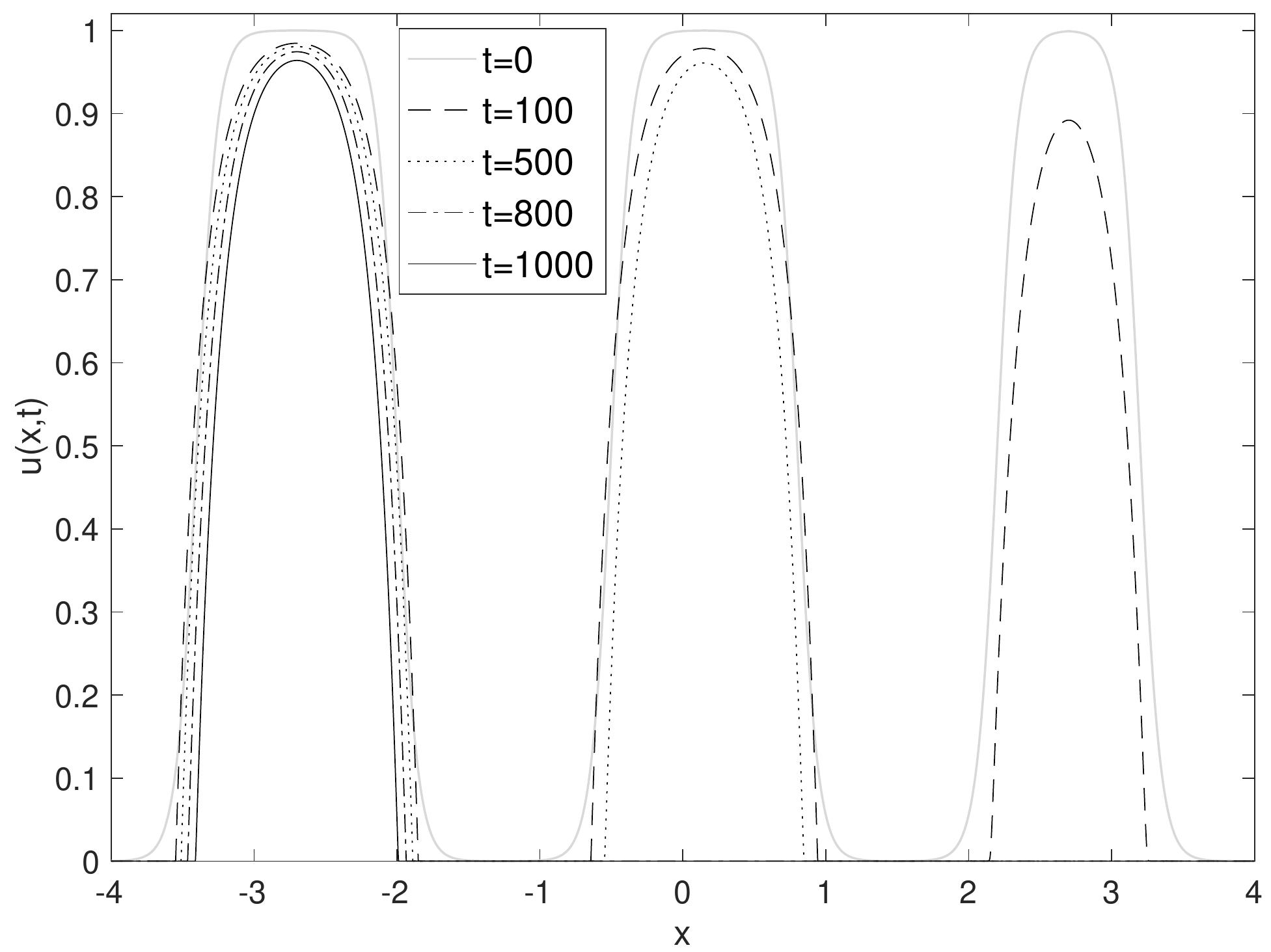}}
\subfigure[$\e=0.06$]{\label{fig:deg2-eps006}\includegraphics[scale = 0.4]{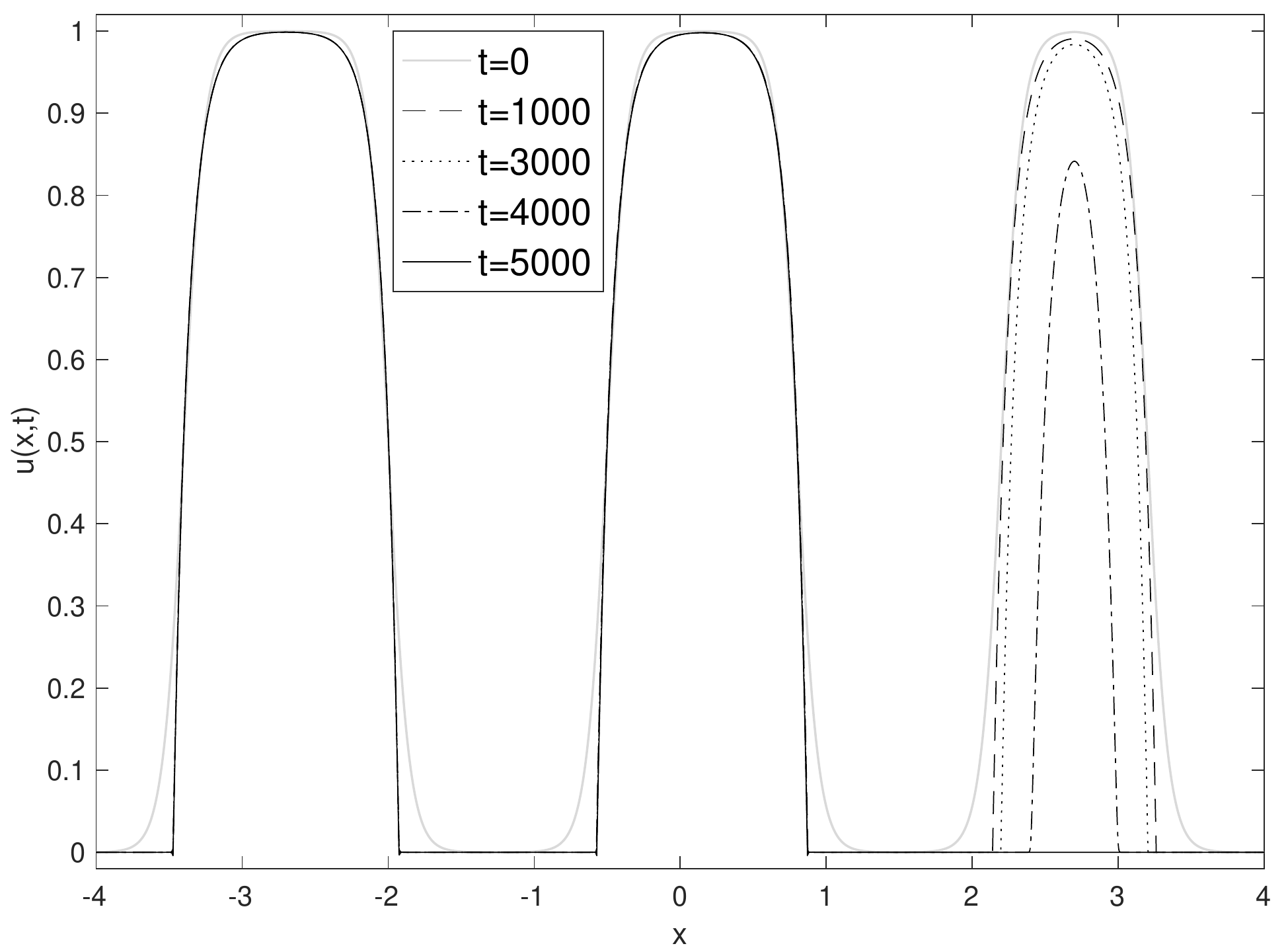}}
\caption{Numerical solutions to \eqref{eq:D-model}-\eqref{eq:Neu}-\eqref{eq:initial} with $D(u)=u^2$, $f(u)=u(u-\frac23)(u-1)$ and two different values for $\e$.
The initial datum has 6 transitions located at the same positions of Figures \ref{fig:classic}-\ref{fig:Mullins}-\ref{fig:exp}-\ref{fig:f-deg}.}
\label{fig:D-deg}
\end{figure}

\section{Discussion}\label{sec:discussion}
In this paper we have rigorously proved the emergence and persistence of metastable structures for the one-dimensional, 
generalized Allen-Cahn equation \eqref{eq:D-model} with phase-dependent diffusivity coefficient. 
Motivated by physical models such as the Mullins diffusion for thermal grooving and the exponential function for diffusion profiles in metal alloys, 
we have assumed that the diffusivity is strictly positive and uniformly bounded below. 
In this fashion, we have extended the previous results on metastability to the case of reaction-diffusion problems where the diffusion is nonlinear and strictly positive. 
The method of proof is based on the energy approach by Bronsard and Kohn \cite{BrKo90} 
and the energy considered in this paper has the form of a generalized effective Ginzburg-Landau functional (see \eqref{eq:energy}). 
There is an innovative work by Otto and Reznikoff \cite{OtRe07} 
that establishes the phenomenon of dynamic metastability for gradient flow equations associated to a large class of energy functionals. 
In this paper, however, the evolution equation \eqref{eq:D-model} \emph{is not the $L^2$-gradient flow of the energy \eqref{eq:energy}} 
(see, for example, Cirillo \emph{et al.} \cite{CISc16}). 
Actually, equation \eqref{eq:D-model} is \emph{not} the $L^2$-gradient flow of any anisotropic energy functional of Ginzburg-Landau type of the form
\[
E_\e[u] = \int_a^b \left\{ \frac{\e^2}{2}\Phi(u, u_x) + F(u) \right\} \, dx  
\]
as the dedicated reader may easily verify. 
Hence, our result does not enter into the framework of the analysis in \cite{OtRe07}. 
We would like to emphasize that \eqref{eq:energy} is the right  energy functional to study the metastable dynamics of the solutions to \eqref{eq:D-model}. 
Our analysis shows that the energy method is applicable to evolution equations beyond the class of gradient flows 
and that the phenomenon of metastability holds for more general reaction-diffusion models. 
We regard these as the main contributions of this paper.

In addition, we have presented the output of numerical simulations for the IBVP problem \eqref{eq:D-model}-\eqref{eq:Neu}-\eqref{eq:initial} 
with different diffusion profiles satisfying the assumptions of this paper 
(namely, the Mullins \eqref{eq:MullinsD} and exponential \eqref{eq:expDiff} diffusion functions), which verify the analytical results. 
We have performed numerical simulations in the case of degenerate diffusion coefficients (for which $D(\alpha) = 0$ or $D(\beta) = 0$) as well. 
These simulations show that the interface layers do not exhibit exponentially slow motion. 
In other words, we have provided numerical evidence that the non-degeneracy conditions on $D$ are necessary for metastability, 
the former understood as exponentially slow motion of interface layers. 
It is to be observed that degenerate diffusions are of interest within the physics community.  
For example, in some binary alloys the diffusivity seems to be zero outside a relatively narrow interfacial band (cf. \cite{TaCa94,EllGa96}), 
that is, $D$ is zero outside the grain boundary including, e.g., the pure phases $u = \pm 1$, and positive inside. 
A function of the form
\[
D(u) = D_0 (1 - u^2), \qquad u\in [-1,1],
\]
can be justified under thermodynamic considerations (see Taylor and Cahn \cite{TaCa94} and the references therein). 
Another example is the aforementioned porous medium type diffusion.

In view of our numerical results and the theoretical ones of \cite{BetSme2013}, 
where the authors consider the classical Allen--Cahn equation \eqref{eq:Al-Ca1d} with \emph{degenerate} $f$, 
we conjecture that the interface motion in the diffusion degenerate case is \emph{algebraically} slow 
and that an energy bound of the form \eqref{eq:sharpest} is not feasible. 
This is an open problem that warrants future investigations.

\section*{Acknowledgements}

The work of R. G. Plaza was partially supported by DGAPA-UNAM, program PAPIIT, grant IN-100318.

\bibliography{riferimenti}

\bibliographystyle{newstyle}


\end{document}